\documentclass[12pt]{amsart}
\usepackage{a4wide}
\usepackage[open, depth=2]{bookmark}
\usepackage{amsmath}
\usepackage{amsfonts}
\usepackage{amssymb}
\usepackage{array}
\usepackage{nicematrix}
\usepackage{bbold}
\usepackage{mathrsfs}
\usepackage{pb-diagram}
\usepackage{color}
\usepackage{dynkin-diagrams}
\usepackage{epstopdf}
\usepackage{enumerate}
\usepackage{comment}
\usepackage{mathdots}
\usepackage{bbm}
\usepackage{tikz}
\usepackage{tkz-euclide}
\usepackage{cleveref}
\usepackage{stmaryrd}
\usepackage{epstopdf}
\newtheorem{lemma}{Lemma}[subsection]
\newtheorem{remark}[lemma]{Remark}
\newtheorem{theorem}[lemma]{Theorem}
\newtheorem{corollary}[lemma]{Corollary}

\newtheorem{proposition}[lemma]{Proposition}

\usepackage{tikz, tikz-cd}
\usetikzlibrary{matrix}
\usetikzlibrary{shapes}
\usetikzlibrary{arrows}
\usetikzlibrary{calc,3d}
\usetikzlibrary{decorations,decorations.pathmorphing}
\usetikzlibrary{through}

\usepackage{xcolor}

\newtheorem{theoremintro}{Theorem}

\usepackage[style=alphabetic,maxnames=99,maxcitenames=99, maxalpha names=99, sorting=nyt]{biblatex}

\newcommand{\supp}{\operatorname{supp}}
\newcommand{\C}{\mathbb{C}}
\newcommand{\Seg}{\operatorname{Seg}}

\newcommand{\m}{\mathfrak{m}}
\newcommand{\n}{\mathfrak{n}}

\newcommand{\Z}{\mathbb{Z}}

\newcommand{\Tab}{\mathrm{Tab}}

\newcommand{\Hom}{\mathrm{Hom}}
\newcommand{\Rep}{\mathrm{Rep}}
\newcommand{\Irr}{\mathrm{Irr}}

\newcommand{\rres}{\mathrm{Res}}

\newcommand{\iind}{\mathrm{Ind}}

\newcommand{\res}{\mathrm{res}}

\newcommand{\ch}{\mathrm{ch}}
\newcommand{\seg}{\mathrm{Seg}}

\newcommand{\tab}{\mathrm{Tab}}

\newcommand{\KK}{\mathcal{K}(\underline{\m},s)}
\newcommand{\JJ}{\mathcal{J}(\underline{\m})}

\newcommand{\qch}{\boldsymbol{\chi^\infty}}
\newcommand{\qchp}{\boldsymbol{\chi^\infty_+}}
\newcommand{\qchN}{\boldsymbol{\chi^N}}
\newcommand{\qchNp}{\boldsymbol{\chi^N_+}}

\usepackage[foot]{amsaddr}

\begin{document}

\title{On reciprocal characters and the quantum affine Schur--Weyl duality}

\begin{abstract}
We identify the dominant part of the Frenkel--Reshetikhin $q$-character with a natural invariant arising from the Langlands/Zelevinsky parameterization for affine Hecke algebras. 

We introduce the reciprocal character of a module over a $GL_n$-type affine Hecke algebra, defined in terms of multiplicities within parabolic restriction. The main theorem claims that the reciprocal character matches, under quantum affine Schur--Weyl duality, with the dominant $q$-character for finite-dimensional modules over quantum affine algebras.

This result gives a type $A$ realization of the Nakajima expectation that the dominant monomials in the $q$-character should play the role of monomial-basis coordinates in Lusztig's framework for finite quantum groups.  

Indeed, under the affine Hecke categorification of $U_q(\mathfrak{sl}_\infty)^+$, we prove that the reciprocal character is the specialization at $q=1$ of the coordinate map attached to a monomial basis. As a consequence, dominant $q$-character multiplicities for simple (or standard) modules are described by transition coefficients between monomial and canonical (or PBW) bases. 

Our methods rely on the development of explicit tableau-counting formulas for such dominant multiplicities, or equivalently for the reciprocal characters of standard modules over affine Hecke algebras.






\end{abstract}

\author{Maxim Gurevich}
	\address{Department of Mathematics, Technion -- Israel Institute of Technology, Haifa, Israel.}
	\email{maxg@technion.ac.il}
	\thanks{This research is supported by the Israel Science Foundation (Grant Number: 2542/25).} 
	\author{Angelina Vargulevich}
	\address{Department of Mathematics, Technion -- Israel Institute of Technology, Haifa, Israel.}
	\email{angelinav@campus.technion.ac.il}
    \date{\today}
	
	\maketitle


\section{Introduction}

A central invariant in the finite-dimensional representation theory of quantum affine algebras is the $q$-character, as introduced in \cite{MR1745260}. Early on, its significance was amplified in the influential work of Nakajima \cite{nakajima-annals}, where deformations of the $q$-character were applied to produce a quantum affine analogy to Lusztig's geometric categorification approach \cite{lusztig-canon} for quantum groups of finite type.

In the exposition of \cite{nakajima-annals}, it is suggested that the role of the \textit{monomial bases} in Lusztig’s framework (or, alternatively, of Feigin’s map as formulated in \cite{reineke-fei}) persists in the quantum affine setting, where it is proposed to be played by the \textit{dominant part} of the $q$-character.


Another cornerstone of the quantum affine theory is the Schur–Weyl duality, as developed in \cite{Chari-Pressley,MR1302015}. This variant of the duality is realized via a family of exact functors $\mathcal{F}_{N,n}$ with favorable properties, which send modules over an affine Hecke algebra $H_{n,v}$, of type $\mathrm{GL}_n$ and a (non-root of unity) deformation parameter $v\in \C^\times$, to modules over $U_{v}\left(\widehat{\mathfrak{sl}_{N+1}}\right)$.

Our work achieves a unified resolution of two closely related objectives:

\begin{itemize}
    \item We construct a native Hecke-algebraic invariant, through Langlands formalism, that mirrors (Theorem \ref{thm:Aintro}) the dominant part of the $q$-character through the Schur--Weyl duality. 
    \item We prove a precise formulation (Theorem \ref{thm:introC}) of the Nakajima expectation on the quantum affine analogue of Lusztig's monomial basis, in Lie type $A$.

\end{itemize}

A major portion of this work deals with the representation theory of the affine Hecke algebras $H_{n,v}$, a theory conveniently situated at the intersection of several themes: Beyond its role in Schur–Weyl duality, the category of $H_{n,v}$-modules is well known to encode aspects of the representation theory of $p$-adic groups, and to appear naturally in categorifications of quantum groups of finite type.


The connection with $p$-adic groups allows us to take a viewpoint of the Langlands quotient formalism (as in \cite{KL}) for the setting of $H_{n,v}$-modules, where it admits a combinatorial realization as the Zelevinsky classification. From this perspective, we introduce a notion of a \textit{reciprocal character} of a module, which will serve as the sought-after invariant that mirrors the quantum affine $q$-character.



Our methods are primarily combinatorial. A key by-product of our approach are explicit tableau-counting formulas for the values (Theorem \ref{thm:introD}) of both the dominant part of the $q$-character and the reciprocal character, when evaluated on standard modules on the two sides of the Schur–Weyl duality. We also show that these formulas are equivalent to the full transition matrix between a PBW basis and a monomial basis of the positive part $U_v(\mathfrak{sl}_k)^+$ of the finite-type quantum group. 

\subsection{Setup}

We consider the Grothendieck ring $G[N]$ of the monoidal category of type $1$ modules over the Hopf algebra $U_{v}\!\left(\widehat{\mathfrak{sl}_{N+1}}\right)$.\footnote{We now fix a non-root of unity $v\in \C^\times$, while retaining the standard terminology of “$q$-character”. }

\subsubsection{Dominant $q$-character}

The classification by Drinfeld polynomials (Proposition \ref{prop:poly-ring}) yields an identification
\begin{equation}\label{eq:drinfeld}
G[N]\cong \Z[\mathscr{D}_{N}]
\end{equation}
of the Grothendieck ring with the (commutative) polynomial ring generated by the variables $\mathscr{D}_{N} = \{Y(i,a)\}_{i,a}$. These variables correspond to fundamental simple modules, indexed by a choice of simple root $i\in \{1,\ldots,N\}$ and a spectral parameter $a\in \C^\times$.

The $q$-character of a module $V$ records the dimensions of generalized eigenspaces $\{\dim(V_{\m})\}_{\m}$ with respect to a certain commutative subalgebra. The weights $\m = Y(i_1,z_1)^{r_1}\cdots Y(i_l,z_l)^{r_l}$\footnote{In the body of this work, an exponential notation $\exp(\m)$ will be used to denote monomials in a polynomial ring. For ease of presentation, we write such monomials as $\m$ throughout the introduction section.} (with $r_1,\ldots,r_l\in \Z$) are identified with Laurent monomials in same variables $\mathscr{D}_N$, a set we denote as $\Z^{\mathscr{D}_N}$. Thus, an injective ring homomorphism
\[
\chi^N: G[N] \hookrightarrow \Z[\mathscr{D}_{N}^{\pm1}]\;,\qquad 
\chi^N(V)  = \sum_{\m\in \Z^{\mathscr{D}_N}}\dim(V_{\m}) \;\m,
\]
is obtained, whose codomain is a Laurent polynomial ring.


A weight is called \emph{dominant} when the corresponding monomial lies in the polynomial subring, i.e.\ when $r_1,\ldots,r_l\geq 0$. The dominant part of the $q$-character is then defined by truncation:
\[
\chi^N_+: G[N] \to \Z[\mathscr{D}_N]\;,\qquad 
\chi^N_+(V) = \sum_{\m\in \Z_{\geq0}^{\mathscr{D}_N}}\dim(V_{\m}) \;\m.
\]

By \cite[Theorem 3.5(3)]{nakajima-annals}, the map $\chi_+^N$, although not multiplicative, is an additive isomorphism. In view of \eqref{eq:drinfeld}, it may therefore be regarded as a locally finite additive automorphism of the polynomial ring $\Z[\mathscr{D}_N]$.


\subsubsection{Bernstein--Zelevinsky ring}\label{sect:bz-intro}

A protagonist of our setting is the polynomial ring, whose variables are indexed by the set of \emph{segments} $\widetilde{\Seg}$. These are formal pairs $[x,y]$, with $x,y\in \C^\times$ related by $y= x v^{2k}$ for integer $k\geq 0$.


The Zelevinsky classification yields an identification (Section \ref{sect:zel})
\begin{equation}\label{eq:bz-ring}
\mathbf{R}:= \bigoplus_{n=0}^\infty G(H_{n,v}) \cong\Z\left[\widetilde{\seg}\right]
\end{equation}
of rings, where $\mathbf{R}$ consists of a sum of affine Hecke algebra Grothendieck groups equipped with a parabolic induction product, while the right-hand side is a polynomial ring with $\widetilde{\seg}$ as the set of variables. 

From another perspective, there are natural quotient maps
\[
P_N: \Z\left[\widetilde{\seg}\right] \twoheadrightarrow \Z[\mathscr{D}_N]
\]
between the polynomial rings, given by
\[
\Delta=\left[av^{-k},av^{k}\right]\in \widetilde{\seg}\,\mapsto\, P_N(\Delta) = \left\{\begin{array}{ll}  Y(k+1, a) & k<N \\
1 & k=N \\ 0 & k>N
\end{array} \:\in \Z[\mathscr{D}_N] \right.\;,
\]
which relate the two realizations. Combining the parametrizations of simple modules by Drinfeld polynomials \eqref{eq:drinfeld} and by Zelevinsky/Langlands data \eqref{eq:bz-ring}, it is known (explicated in Theorem \ref{thm:commute-qasw})that $P_N$ coincides with the semisimplification of the combined Schur--Weyl duality functor $\bigoplus_{n=0}^\infty \mathcal{F}_{N,n}$.


We now state the main result of this work. It regards an additive isomorphism 
\[
\ch^{\otimes}:\mathbf{R}\to \Z\!\left[\widetilde{\seg}\right]
\]
that will be described below \eqref{eq:rec-ch}.

\begin{theoremintro}\label{thm:Aintro}
Let $V$ be a finite-dimensional $H_{n,v}$-module and $W= \mathcal{F}_{N,n}(V)$ be the finite-dimensional $U_{v}(\widehat{\mathfrak{sl}_{N+1}})$-module obtained from $V$ by quantum affine Schur--Weyl duality.

Then, the relation
    \[
P_N(\ch^{\otimes}(V)) = \chi^N_+(W)\in \Z[\mathscr{D}_N]
    \]
holds, between the dominant part of the quantum affine $q$-character and the reciprocal affine Hecke character.

\end{theoremintro}

\subsection{The reciprocal character}

The notion of a character for affine Hecke algebra modules is multifaceted. In some strict senses, a character would amount to recording generalized eigenspace multiplicities (see \cite[Corollary 3.15]{antor-okada}). While this invariant plays a prominent role in our discussion (the map $\ch$ in Section \ref{sect:ch}), it also contains a significant amount of redundancy and is not directly suited for comparison with the quantum affine $q$-character.


When moving by standard categorical equivalences from a $H_{n,v}$-module to a smooth $\mathrm{GL}_n(\mathbb{Q}_{v^2})$-representation (available when $v^2$ is a prime integer), local Langlands reciprocity places the affine Hecke setting in close analogy with classical highest-weight theory, where character-theoretic questions are governed by the transition coefficients between standard and simple modules.


\subsubsection{Langlands quotient}
To each monomial $\m\in \Z_{\geq0}^{\widetilde{\seg}}$ in $\Z\left[\widetilde{\seg}\right]$, that is a \textit{multisegment}, Zelevinsky constructs a standard $H_{n_{\m},v}$-module $\zeta(\m)$. These modules form a $\Z$-basis for $\mathbf{R}$ and realize the isomorphism in \eqref{eq:bz-ring}. Each of them admits a unique simple quotient $Z(\m)$, and the resulting correspondence constitutes the cornerstone Zelevinsky classification of simple $H_{n,v}$-modules.

Since standard modules are comparatively well understood and readily constructed, character theory often boils down to understanding the expansion 
\[
[V] = \sum_{\m\in \Z_{\geq0}^{\widetilde{\seg}}\;:\; n_{\m}=n} m(V,\m) \cdot [\zeta(\m)]\in \mathbf{R}
\]
of a given $H_{n,v}$-module $V$ in terms of the standard basis. Such expansions, and in particular the integer coefficients $m(Z(\m),\zeta(\n))$, are known to admit geometric descriptions and are given by values (at $1$) of Kazhdan--Lusztig polynomials of finite type $A$.

\subsubsection{Indicator modules}

Let us recount that a composition $n= n_1+\ldots + n_r$ gives a parabolic embedding $H_{\underline{n},v}:=H_{n_1,v}\otimes \cdots\otimes H_{n_r,v}\hookrightarrow H_{n,v}$ of algebras, together with a corresponding adjoint pair of induction and restriction functors $\iind_{\underline{n}},\; \rres_{\underline{n}}$ between module categories.

For each standard module, we present (Section \ref{sect:indicators}, complemented by Proposition \ref{prop:non-canon-ch}) a straightforward recipe for a simple \textit{indicator module} $\zeta(\m)^{\otimes}$ over a suitable parabolic subalgebra $H_{\underline{n_{\m}},v}<H_{n_{\m},v}$, so that 
\[
\zeta(\m) = \iind_{\underline{n_{\m}}}(\zeta(\m)^{\otimes})
\]
holds, for each multisegment $\m$. 

We consequently define the reciprocal character of a $H_{n,v}$-module $V$ to be the polynomial
\begin{equation}\label{eq:rec-ch}
\ch^\otimes(V) = \sum_{\m\in \Z_{\geq0}^{\widetilde{\seg}}\;:\; n_{\m}=n} \left[ \rres_{\underline{n_{\m}}}(V)\,:\, \zeta(\m)^{\otimes}\right]\cdot \m \in \Z\left[\widetilde{\seg}\right]\;,
\end{equation}
which counts the Jordan--H\"older multiplicities of indicator modules inside the (possibly non semi-simple) parabolic restrictions of $V$.

The information encoded in the reciprocal character of a module has been immensely valuable in previous works of the first author \cite{gur-qinv,gur-jlms}, where it was used to analyze subtle features of the monoidal structure in $\mathbf{R}$. When transported to the setting of $p$-adic groups, restriction functors correspond to Jacquet functors, while the indicator representations considered in \cite[Definition 3.1]{gur-decomp} match the indicator modules appearing in our framework. The concept of indicator representations has recently featured \cite[Remark 3.6]{atobe-ming25} in the analysis of the unitary dual of reductive $p$-adic groups of classical type. 

It has also been established (see, e.g., \cite{gur-jems}) that the monoidality of the Schur--Weyl duality functors $\mathcal{F}_{N,n}$ provides an effective bridge for importing results from quantum affine representation theory - often derived using the $q$-character invariant - into spectral problems in the domain of $p$-adic groups. From this perspective, Theorem~\ref{thm:Aintro} may be viewed as a neat completion of this circle of ideas.

\subsection{Finite quantum group perspective}

In classical Lie theory, the positive part of the universal enveloping algebra $U(\mathfrak{sl}_k)^+$ is realized as a quotient of the free associative algebra $R_k$ on the simple $\mathfrak{sl}_k$-roots, modulo Serre relations. In the Drinfeld--Jimbo quantum group setting (Section \ref{sect:qgroups}), this picture deforms to a quotient $\mathbb{C}(q)\otimes R_k\twoheadrightarrow U_q(\mathfrak{sl}_k)^+$ by the quantum Serre relations that involve the formal parameter $q$. 


We recall some of the essentials of Lusztig’s theory of the canonical basis, for the case of $U_q(\mathfrak{sl}_{\infty})^+ = \varinjlim_k U_q(\mathfrak{sl}_k)^+$, ignoring routine issues related to the limit.

In this setting, let $R(q)$ denote the free $\mathbb{Q}(q)$-algebra on a countable set of generators. Thus, $R(q)=\mathrm{span}_{\mathbb{Q}(q)}\{w\}_{w\in \mathcal{W}}$ is an algebra spanned by a basis of words $\mathcal{W}$ in a countable alphabet $\{\alpha_i\}_{i\in \Z}$, with multiplication given by concatenation. A distinguished sub-$\mathbb{Z}[q,q^{-1}]$-algebra $R[q] \subset R(q)$ is generated by the divided power normalizations $w' = \kappa_w^{-1} w$ (see Section \ref{sect:qgroups-word}), for each $w \in \mathcal{W}$.

The quotient $R[q]\twoheadrightarrow \mathbf{f}$ by the ideal generated by the quantum Serre relations \eqref{eq:quantum-serre} (of type $A_\infty$) is known as \textit{Lusztig’s algebra}.

This resulting $\Z[q,q^{-1}]$-free module $\mathbf{f}$, together with its dual module $\mathbf{f}^\ast$, forms the natural setting for geometric categorifications and for the axiomatic construction of the canonical basis $\mathscr{B}\subset \mathbf{f}$. This basis also specializes to a $\mathbb{Z}$-basis of the abelian group 
\[
(\mathscr{B})_{q\to1}\subset (\mathbf{f})_{q\to 1}:= \mathbf{f}\otimes_{\Z[q,q^{-1}]} \mathrm{ev}_1\;,
\]
where $\mathrm{ev}_1=\Z$ is the $q=1$ evaluation $\Z[q,q^{-1}]$-module.

\subsubsection{Monomial basis}

One avenue of interest in this framework is a description of subsets $\mathcal{M} \subset \mathcal{W}$, for which $\{w'\}_{w \in \mathcal{M}}$ projects to a $\mathbb{Z}[q,q^{-1}]$-basis of $\mathbf{f}$. Such bases are called monomial bases, reflecting the interpretation of words as monomials in $R(q)$.

A systematic construction of monomial bases, depending on a choice of reduced decomposition of the longest Weyl group element, was already given in \cite[Section 7.8]{lusztig-canon}, where it played a key role in the construction of the canonical $\mathscr{B}$. This monomial construction was later formalized by Reineke \cite{reineke-fei}, who clarified its quiver-theoretic origins and its tight connection to the Feigin map realization of quantum groups (\cite{ioh-mal,joseph-feigin}).


Among subsequent works, \cite{Antor23} further explicated an algorithmic approach in which monomial bases serve as an intermediate step between $\mathscr{B}$ and PBW-type bases of $\mathbf{f}$. In our context it is notable that Nakajima \cite{nakajima-annals} interpreted bases of simple and standard modules in quantum affine Grothendieck rings, such as $G[N]$, as an analogous basis transition problem, where the $q$-character map took the role of the intermediary.


For our purposes, we fix a single choice of a monomial basis, given by a fixed collection of words $\mathcal{M} = \{w(\m)\}_{\m} \subset \mathcal{W}$. 

Those words are indexed by integral multisegments, that is, by monomials $\m\in \Z_{\geq0}^{\seg(1)}$ in a subring $\mathbb{Z}[\seg(1)]<\Z\left[ \widetilde{\seg}\right]$, where $\seg(1)=\{[x,y]: x,y\in v^{2\Z}\}\subset \widetilde{\seg}$. Indeed, the integral segments naturally correspond to positive roots for the $A_{\infty}$ Cartan datum.

Notably, our choice of a monomial basis traces back to the Berenstein--Zelevinsky treatment \cite[Corollary 3.6(b)]{MR1237826} of string bases in quantum groups. 

The projection of (the normalizations of) $\mathcal{M}$ provides a $\Z[q,q^{-1}]$-basis $\mathscr{M}=\{m_{\m}\}_{\m\in \Z_{\geq0}^{\seg(1)}}$ for $\mathbf{f}$, and thus a dual monomial $\Z[q,q^{-1}]$-basis $\{m_{\m}^\ast\}_{\m\in \Z_{\geq0}^{\seg(1)}}$ for $\mathbf{f}^\ast$. We take note of the resulting coordinate map
\[
C_{\mathscr{M}}: \mathbf{f}^\ast  \overset{\sim}{\rightarrow} \Z[q,q^{-1}]\otimes\Z[\seg(1)]\;,\quad C_{\mathscr{M}}(m_{\m}^\ast) = 1\otimes\m\;,\forall\m \in \Z_{\geq0}^{\seg(1)}\;.
\]

\subsubsection{Affine Hecke categorification}
Another key feature in type $A$ studies is the categorification of the $q=1$ specialization of Lusztig’s algebra via the representation theory of $H_{n,v}$. This phenomenon stems from underlying geometric parallels between Schubert varieties and quiver representation varieties, and was established through different approaches by Ariki \cite{ariki-decomp} or Grojnowski \cite{grojn}. An additional perspective is provided by quiver Hecke algebra categorifications (\cite{kl09,rouq12,vv11}) and their identification with affine Hecke algebras in type $A$ (see \cite[Section 3.2.1]{gur-jlms}).

This categorification is encapsulated by a canonical injection
\begin{equation}\label{eq:categor}
\Phi: (\mathbf{f}^\ast)_{q\to1} \hookrightarrow \mathbf{R}
\end{equation}
from the specialization $(\mathbf{f}^\ast)_{q\to1} =\mathbf{f}^\ast\otimes_{\Z[q,q^{-1}]}\mathrm{ev}_1$ of the $\mathbb{Z}[q,q^{-1}]$-module into the Bernstein–Zelevinsky ring. The map $\Phi$ enjoys several favorable properties: It identifies the specialization of the dual canonical basis $(\mathscr{B}^\ast)_{q\to1}\subset (\mathbf{f}^\ast)_{q\to1} $ with the isomorphism classes of simple modules in the integral blocks of $\mathbf{R}$. Likewise, it sends a specialization of the dual of a certain PBW-type basis $\mathscr{E} = \{e_{\m}\}\subset \mathbf{f}$ (see Section \ref{sect:pbw}), again indexed by integral multisegments $\m\in \Z_{\geq0}^{\seg(1)}$, to the set of integral standard modules, i.e. $\Phi(e_{\m}^\ast\otimes 1) = [\zeta(\m)]$.

Furthermore, $\Phi$ is multiplicative with respect to a natural ring structure on $\mathbf{f}^\ast$, although this feature will not be used in our treatment.


\begin{theoremintro}\label{thm:introB}[Theorem \ref{thm:rec-ch}]
When specialized at $q=1$, the coordinate map $C_{\mathscr{M}}$ with respect to the dual monomial basis for $\mathbf{f}^\ast$ coincides with the reciprocal character map for affine Hecke algebras, under the canonical categorification of Lusztig's algebra.

 Namely, the diagram   
 \[
\begin{diagram}
\dgARROWLENGTH=2em
\node{ (\mathbf{f}^\ast)_{q\to 1} } \arrow{s,t}{\Phi }  \arrow{se,t}{ (C_{\mathscr{M}})_{q\to1}}\\
\node{\mathbf{R} } \arrow{e,t}{ \Large \ch^\otimes } \node{ \Z\left[\widetilde{\seg}\right]  }
\end{diagram}\;
\]
commutes.

In concrete terms, whenever a $H_{n,v}$-module $V$ corresponds $[V]= \Phi(a_V\otimes 1)$ to the specialization of a module homomorphism $a_V: \mathbf{f}\to \Z[q,q^{-1}]$ under the categorification map \eqref{eq:categor}, we have
\[
\ch^\otimes(V) = \sum_{\m\in \Z_{\geq0}^{\seg}} a_V(m_{\m})(1) \cdot \m\in \Z[\seg(1)]\;.
\]


\end{theoremintro}

The combination of Theorems \ref{thm:Aintro} and \ref{thm:introB} provides the following direct link between the dominant part of the $q$-character and a monomial basis in Lusztig's algebra. For sake of exposition, let us formulate the corollary in a form that is independent of the Schur--Weyl duality functor construction.

\begin{theoremintro}\label{thm:introC}

Let $W$ be either a simple, or a standard, finite-dimensional type $1$ module over $U_{v}\!\left(\widehat{\mathfrak{sl}_{N+1}}\right)$, and $\m_0\in \Z_{\geq0}^{\seg(1)}$ a choice of an integral multisegment, for which $P_N(\m_0)\in \Z_{\geq0}^{\mathscr{D}_N}$ is the parameterizing Drinfeld polynomial of $W$.

Let $b_0\in \mathscr{B}\subset\mathbf{f}$ be the canonical basis element, whose dual $b_0^\ast: \mathbf{f}\to \Z[q,q^{-1}]$ satisfies $\Phi(b_0^\ast\otimes 1) = Z(\m_0)$ through the Zelevinsky classification and the categorifcation map \eqref{eq:categor}.

Let us write the full transition matrices of Laurent polynomials 
\[
h_{\m,b}, s_{\m,\n}\in \Z[q,q^{-1}]\;, \mbox{ with } \;m_{\m} = \sum_{b\in \mathscr{B}} h_{\m,b}\,b=\sum_{\n\in \Z_{\geq0}^{\seg(1)}} s_{\m,\n}\,e_{\n}\in \mathbf{f}\;,\forall \m \in \Z_{\geq0}^{\seg(1)}\;,
\]

between the monomial $\mathscr{M}$, canonical $\mathscr{B}$ and PBW-type $\mathscr{E}$ bases of the free module $\mathbf{f}$.

Then, 
\begin{equation}\label{eq:C1}
\chi^N_+(W) = \left\{ \begin{array}{lc}\sum_{\m\in \Z_{\geq0}^{\seg(1)}} h_{\m,b_0}(1) \cdot P_N(\m) & \mbox{simple } W \\
\sum_{\m\in \Z_{\geq0}^{\seg(1)}} s_{\m,\m_0}(1) \cdot P_N(\m) & \mbox{standard } W
\end{array}\right.\;\in \Z[\mathscr{D}_N]
\end{equation}
holds.

Moreover, whenever a multisegment $\m \in \Z_{\geq0}^{\seg(1)}$ has $P_N(\m)\neq0$, the formula
\begin{equation}\label{eq:C2}
\dim(W_{P_N(\m)}) =\left\{ \begin{array}{lc} h_{\m,b_0}(1)  & \mbox{simple } W \\
 s_{\m,\m_0}(1)  & \mbox{standard } W
\end{array}\right.\;,\;\forall \m\in\Z_{\geq0}^{\seg(1)}
\end{equation}
describes the dimensions of all dominant generalized eigenspaces that are counted by the $q$-character invariant of $W$.

\end{theoremintro}

Let us remark that the integrality condition on the multisegment $\m_0$ in the formulation of Theorem \ref{thm:introC} is imposed for clarity of statement and entails no essential loss of generality. Indeed, any type~$1$ $U_{v}\!\left(\widehat{\mathfrak{sl}_{N+1}}\right)$-module, together with its $q$-character, admits a standard decomposition (Section \ref{sect:qa-integral}) into a tensor product of its integral components.

We finally point out that the identity in \eqref{eq:C1} admits a natural
deformation on both sides. On the quantum affine side, the $q$-character is refined by the so-called $q,t$-character,\footnote{As before, we follow the standard terminology, although the ``$t$-part'' of the invariant corresponds to the parameter denoted by $q$ in our notation.} whose dominant part lies in
$\Z[\mathscr{D}_N,q,q^{-1}]$. On the quantum group side, the transition
coefficients appearing in Theorem \ref{thm:introC} may be retained before
specialization at $q=1$. Thus, it may be conjectured that the present result is a $q=1$ shadow of a richer compatibility between these two deformations.

This perspective also suggests deformed counterparts of Theorems
\ref{thm:Aintro} and \ref{thm:introB}. A natural framework for such questions
may be provided by the graded quiver Hecke algebras and their Schur--Weyl duality functors,
as in \cite{kkk-inentiones}.

A different perspective on the same quantum affine Grothendieck rings is provided by the cluster-algebra approach of Hernandez--Leclerc \cite{HL10,HL16}.  The work of Casbi--Li \cite{CL23}, for example, studies the $q$-character through this celebrated categorification. The relations between the cluster-theoretic picture and the monomial-basis coordinates studied in the present work are a frontier for exploration.

\subsubsection{Example}

We illustrate Theorem \ref{thm:introC} on an example case that often features in the literature, as in \cite[Section 11.4]{zel-main}, \cite[Example 4]{lnt}, \cite[Example 5.2]{Antor23}.

We consider the $\Z[q,q^{-1}]$-module 
\[
\Lambda= \mathrm{span}_{\Z[q,q^{-1}]}\{\overline{w}_j\}_{j=1}^5<\mathbf{f}
\]
that is generated by the projections of the $5$ elements 
\[
w_1= w[1223], w_2= w[2123], w_3 = w[2213], w_4 = w[1232], w_5 =w[1322]\in R[q]\;.
\] 
Here $w[i_1\ldots i_t] = (\alpha_{i_1}\cdots\alpha_{i_t})'\in R(q)$ denotes the divided power normalizations of the corresponding words.

The quantum Serre relations imply $\overline{w_2} = \overline{w_1} + \overline{w}_3$, and $\overline{w_4} = \overline{w_1} + \overline{w}_5$, so that $\Lambda$ is free of rank $3$.

Now consider the multisegments
\[
\begin{array}{l}
\m = [v^{2\cdot1},v^{2\cdot1}] + 2[v^{2\cdot2},v^{2\cdot2}] + [v^{2\cdot3},v^{2\cdot3}]\;, \\
\m' = [v^{2\cdot1},v^{2\cdot2}] + [v^{2\cdot2},v^{2\cdot2}] + [v^{2\cdot3},v^{2\cdot3}]\;, \\
\m'' = [v^{2\cdot1},v^{2\cdot1}] + [v^{2\cdot2},v^{2\cdot2}] + [v^{2\cdot2},v^{2\cdot3}]
\end{array}\; \in \Z_{\geq0}^{\seg(1)}\;.
\]
The PBW basis elements
\[
\begin{array}{l}
e_{\m} = \overline{w[1223]} =  \overline{w}_1\;,\\
e_{\m'} = \left(\overline{w[2]}\cdot \overline{w[1]} - q^{-1}\overline{w[1]}\cdot\overline{w[2]} \right)\overline{w[23]} = \overline{w_2} - q^{-1}(q+q^{-1}) \overline{w_1}, \\
e_{\m''} = \overline{w[12]}\left(\overline{w[3]}\cdot \overline{w[2]} - q^{-1}\overline{w[2]}\cdot\overline{w[3]} \right) = \overline{w_4} - q^{-1}(q+q^{-1}) \overline{w_1}
\end{array}\;\in \mathscr{E}\;,
\]
constitute a basis for $\Lambda$. Subsequently, the elements
\[
\left\{\begin{array}{l}
b_{\m} = e_{\m} =  \overline{w_1}\;,\\
b_{\m'} =   e_{\m'}  + q^{-2}e_{\m} = \overline{w_2} - \overline{w_1} = \overline{w_3} , \\
b_{\m''} =   e_{\m''}  + q^{-2}e_{\m} = \overline{w_4} - \overline{w_1} = \overline{w_5} \\
\end{array}\right\}\; = \Lambda\cap \mathscr{B}\;,
\]
form the corresponding part of the canonical basis, while our chosen monomial basis gives
\[
\left\{\begin{array}{l}
m_{\m} = \overline{w_1}= b_{\m} = e_{\m}\;,\\
m_{\m'} =  \overline{w[221]}\cdot \overline{w[3]} = \overline{w_3} = b_{\m'} = e_{\m'}  + q^{-2}e_{\m} , \\
m_{\m''} =  \overline{w[1]}\cdot \overline{w[2]}\cdot \overline{w[32]}=  \overline{w_4} =   \overline{w_5} + \overline{w_1} =  b_{\m''}+ b_{\m} =   e_{\m''}+ (1+q^{-2})e_{\m} \\
\end{array}\right\}\; = \Lambda\cap \mathscr{M}\;.
\]

We now specialize to the quantum affine setting with $N=1$, where the map $P_1$ sends the above multisegments to the following Drinfeld monomials:
\[
P_1(\m) = Y(1,v^4)^2 Y(1,v^2) Y(1,v^6),\; P_1(\m') = Y(1,v^4) Y(1,v^6),\; P_1(\m'') = Y(1,v^4)Y(1,v^2)\in \Z[\mathscr{D}_1]\;.
\]
The corresponding standard modules $S, S',S''\in \mathcal{C}_1$ admit the $q$-characters
\[
\begin{array}{l}
\chi^{1}\bigl(S\bigr)
= \bigl(Y(1,v^4) + Y(1,v^6)^{-1}\bigr)^{\!2}\,
   \bigl(Y(1,v^2) + Y(1,v^4)^{-1}\bigr)\,
   \bigl(Y(1,v^6) + Y(1,v^8)^{-1}\bigr), \\
\chi^{1}\bigl(S'\bigr)
=  \bigl(Y(1,v^4) + Y(1,v^6)^{-1}\bigr)\,
   \bigl(Y(1,v^2) + Y(1,v^4)^{-1}\bigr), \\
\chi^{1}\bigl(S''\bigr)
= \bigl(Y(1,v^4) + Y(1,v^6)^{-1}\bigr)\,
   \bigl(Y(1,v^6) + Y(1,v^8)^{-1}\bigr)
\end{array}\; \in \Z[\mathscr{D}_1^{\pm1}]\;,
\]
which implies that their dominant parts are given as
\[
\begin{array}{ll}
\chi^{1}_+\bigl(S\bigr)&
= P_1(\m) + P_1(\m') + 2P_1( \m'') + 2= \\
 &=  s_{\m,\m}(1) P_1(\m) +  s_{\m',\m}(1) P_1(\m') +s_{\m'',\m}(1) P_1(\m'')+2   \;,\\
\chi^{1}_+\bigl(S'\bigr) &
=  P_1(\m') + 1 = s_{\m',\m'}(1) P_1(\m')+1\;,\\
\chi^{1}_+\bigl(S''\bigr)
&=  P_1(\m'') +1 = s_{\m'',\m''}(1) P_1(\m'')+1
\end{array}\; \in \Z[\mathscr{D}_1]\;.
\]

The constant terms in the dominant $q$-characters arise from the multisegment $\m'''= [v^{2\cdot1}, v^{2\cdot 2}]+ [v^{2\cdot2}, v^{2\cdot 3}]\in \Z_{\geq0}^{\seg}$, with $P_1(\m''')=1$, that provides the monomial basis element $m_{\m'''} = \overline{w[2132]}\in \mathbf{f}$, and the PBW basis element 
\[
\begin{array}{ll}
e_{\m'''} &= \left( \overline{w[21]} - q^{-1}\overline{w[12]} \right)\left( \overline{w[32]} - q^{-1}\overline{w[23]} \right)\\
&= m_{\m'''}- q^{-1}\overline{w_2} - q^{-1}\overline{w_4} + q^{-2}(q+q^{-1})\overline{w_1} \\
&= m_{\m'''} - q^{-1}e_{\m'}  -q^{-1}e_{\m''}  - (q^{-1}+ q^{-3})e_{\m}\;
\end{array}\;.
\]
Thus, $s_{\m''',\m''}(1) =s_{\m''',\m'}(1) =1$ and $s_{\m''',\m}(1)=2$ hold, and a full confirmation of the statement of Theorem \ref{thm:introC} is visible.

We also observe that the identity $m_{\m''} = b_{\m''}+b_{\m}$ implies, by Theorem \ref{thm:introC}, that the simple module $W=\mathcal{F}_N(Z(\m))\in \mathcal{C}_N$ is \textit{non-special}, in the sense of \cite[Definition 10.1]{nakajima-annals}: Besides its highest weight $P_N(\m)$, a dominant weight space $W_{P_N(\m'')}\neq\{0\}$ appears. Indeed, a direct computation may show that $\chi^N_+(W) = P_N(\m)+P_N(\m'')$.



\subsection{Combinatorial weight counting for standard modules}

In the analysis of the introduced map $\ch^\otimes$, we obtain an effective formula for the matrix of integer coefficients
\[
A(\m,\n):= [\rres_{\underline{n_{\n}}}(\zeta(\m))\,:\, \zeta(\n)^\otimes]\;,
\]
indexed by pairs of integral multisegments $\m,\n\in \Z_{\geq0}^{\seg(1)}$. This is the matrix that encodes the full reciprocal characters of standard $H_{n,v}$-modules.

By Theorem \ref{thm:introB}, these coefficients coincide with the specialization at $q=1$ of the (transposed) transition coefficients $s_{\m,\n}\in \Z[q,q^{-1}]$ between the bases $\mathscr{M}$ and $\mathscr{E}$, as defined in the statement of Theorem \ref{thm:introC}, that is, $A(\m,\n) = s_{\n,\m}(1)$.

Our approach rests on the claim that essentially the same formula governs multiplicities in the dominant part of $q$-characters of standard modules in the quantum affine setting. More precisely, Theorem \ref{thm:Aintro} follows, when we prove the identity
\[
\chi_+^N(\mathcal{F}_{N,n_{\m}}(\zeta(\m))) 
= \sum_{\n\in \Z_{\geq0}^{\seg}} A(\m,\n)\, P_N(\n),
\qquad 
\forall \m\in \Z_{\geq0}^{\seg(1)},\ \forall N\geq1.
\]

The coefficients $A(\m,\n)$, as well as their $q$-deformation $s_{\n,\m}$, arise naturally in contexts where monomial bases are employed (\cite{Antor23,shoji2025}). Thus, we consider the combinatorial description obtained here to be of independent interest. We now proceed to describe it in detail.

\subsubsection{Mackey tableaux}\label{sect:intro-tabl}

Suppose now that $\m\in \Z_{\geq0}^{\seg(1)}$ is an integral multisegment, given by a monomial $[v^{2x_1},v^{2y_1}],\ldots,[v^{2x_l},v^{2y_l}]\in \seg(1)$ of segments (on which we fix an arbitrary order). We define the finite set of \textit{bi-tableaux} $\mathcal{A}(\m)$, so that each member $\mathcal{Q}= (a^i_j, b^i_j)_{(i,j)\in I_{\mathcal{Q}}}\in \mathcal{A}(\m)$ consists of an array of integers $a^i_j\in \{x_1,\ldots,x_l\}$, $b^i_j\in \{y_1,\ldots,y_l\}$, that are indexed by $I_{\mathcal{Q}} = \{(i,j)\in \Z\times \Z\:: 1\leq i \leq l,\; 1\leq j\leq k_i\}$ and satisfy
\[
x_i = a^i_{k_i} < \cdots<a^i_1\leq y_i\leq  b^i_{1} < \cdots<b^i_{k_i}\;.
\]
For each such bi-tableau $\mathcal{Q}\in \mathcal{A}(\m)$ and a segment $\Delta = [v^{2x},v^{2y}]\in \seg(1)$, we take note of the statistics
\[
\begin{array}{cll}
t(\mathcal{Q}, \Delta)& = &\#\{(i,j)\in I_{\mathcal{Q}}\::\: (a^i_j,b^i_j)= (x,y)\}\;,
\\
\overleftarrow{t}(\mathcal{Q}, \Delta) &=& \#\{(i,j)\in I_{\mathcal{Q}}\::\: y_i<y,\, (a^i_{j-1},b^i_j)= (x,y)\}\;,
\end{array}
\]
where $a^i_0=y_i+1$ is assumed.

We say that $\mathcal{Q}\in \mathcal{A}(\m)$ is a \textit{dominant} bi-tableau, when the inequalities $t(\mathcal{Q}, \Delta)\geq \overleftarrow{t}(\mathcal{Q}, \Delta)$ hold, for all segments $\Delta\in \Seg$, and denote the set $\mathcal{A}(\m)_+\subseteq \mathcal{A}(\m)$ of dominant bi-tableaux.

For an integral multisegment $\n\in \Z_{\geq0}^{\seg(1)}$, we further denote the power $r(\n,\Delta)\geq0$ with which the segment appears in $\n$, as a monomial in segments. 

\begin{theoremintro}\label{thm:introD}
For any pair of multisegments $\m,\n\in \Z_{\geq0}^{\seg(1)}$, the corresponding reciprocal multiplicity is given by the formula 
\[
A(\m,\n) = \#\{\mathcal{Q}\in \mathcal{A}(\m)_+\;:\: r(\n,\Delta) = t(\mathcal{Q},\Delta) - \overleftarrow{t}(\mathcal{Q},\Delta)\;, \forall \Delta\in \seg\}\;,
\]
which counts the number of dominant bi-tableaux in $\mathcal{A}(\m)$ that satisfy a set of matching conditions that are determined by $\n$.

\end{theoremintro}


\section{Preliminaries}

\subsection{General constructions}

Given a set $A$, we write $\Z^A$ for the set of functions $A\to \Z$ with \textit{finite} support. We treat it as an additive free abelian group whose generator set is $A$. That is, elements will be written as formal combinations $\sum_{a\in A} c_a a\in \Z^A$, with $c_a\in \Z$. In particular, we often treat $A \cong \{1\cdot a\in \Z^{A}\}_{a\in A}$ as a subset of $\Z^A$.

The submonoid $\Z_{\geq0}^A\subset \Z^A$ may be viewed as the collection of multisets of elements in $A$.

We write $\mathcal{W}(A) = \bigsqcup_{k\geq0} A^k$ for the set of words in $A$, treating it as a free (non-commutative) multiplicative monoid with respect to word concatenation. A convention $\{e\} = A^0 \subset\mathcal{W}(A)$ is imposed here.

Let us denote the monoid homomorphism $\mathcal{W}(A)\to \Z_{\geq0}^A$, $w\mapsto\langle w\rangle$, obtained by assigning $\langle (a_1,\ldots,a_k)\rangle:= a_1+\ldots + a_k$, for a word $(a_1,\ldots,a_k)\in A^k$.

For example, with $A=\{a_1, a_2, a_3\}$ and a word $w=(a_1, a_2, a_1, a_3)\in A^4$, we have $\langle w \rangle = 2 a_1 + a_2 + a_3$.

\subsubsection{Tableaux}
We treat elements of $\tab(A):=\mathcal{W}(\mathcal{W}(A))$ as \textit{tableaux in} $A$. A given element $\mathcal{P} = \mathcal{P}_1\cdots \mathcal{P}_r\in \tab(A)$ with $\mathcal{P}_i = x^i_1\cdots x^i_{k_i}\in \mathcal{W}(A)$, $i=1,\ldots, r$, may equivalently be considered as a function
\[
\mathcal{P}: I_{\mathcal{P}}\to A\;,\quad \mathcal{P}(i,j) = x^i_j\;,
\]
where $I_{\mathcal{P}} = \{(i,j)\in \{1,\ldots,r\}\times\Z\,:\, 1\leq j\leq k(i)\}$.

It will be useful to denote the auxillary index set $\widehat{I}_{\mathcal{P}} = \{(i,j)\in I_{\mathcal{P}} \,:\, 2\leq j\}$.

We may also write $\mathcal{P} =(x^i_j)\in \tab(A)$ directly.

We refer to elements of $\Tab(\Z\times\Z)$ as \textit{bitableaux}.

For $A = \{a_1, a_2\}$, an example of a tableau is $\mathcal{P}=((a_1,a_2),(a_1,a_2,a_2),(a_2))\in \tab(A)$, with $I_{\mathcal{P}}=\{(1,1),(1,2),(2,1),(2,2),(2,3),(3,1)\}$ and $\widehat{I}_{\mathcal{P}}=\{(1,2),(2,2),(2,3)\}$.

\subsubsection{Graded multisets}

We say that a set $A$ is \textit{graded} (respectively, \textit{positively graded}), when a fixed decomposition $A = \bigsqcup_{i\in \Z} A_i$ (respectively, $A = \bigsqcup_{i=0}^\infty A_i$) is assumed. For a graded or positively graded set $A$, we have a corresponding group decomposition $\Z^A = \bigoplus_i\Z^{A_i}$ of groups.

For a positively graded set $A$, we also write $A_{\leq N} = \sqcup_{i=0}^N A_i$.

\subsubsection{Integer constructions}

Let $\mathcal{I} = \{\alpha_i\}_{i\in \Z}$ be a formal set indexed by integers, and consider its finite subsets $\mathcal{I}_{\leq k} = \{\alpha_i\}_{i=-k}^k$, for all integers $k\geq0$.

There are natural inclusions $\Z^{\mathcal{I}_{\leq k}}<\Z^{\mathcal{I}}$ and $\mathcal{W}(\mathcal{I}_{\leq k})<\mathcal{W}(\mathcal{I})$.

For an element $\beta = \sum_{i\in \Z} c_i \alpha_i\in \Z^{\mathcal{I}}$, we denote its \textit{height} $|\beta|= \sum_{i\in \Z}c_i\in \Z$, and write $\beta(i) = c_i$.

We denote the sets of words 
\[
\mathcal{W}_n:=\{w\in \mathcal{W}(\mathcal{I}) \::\: |\langle w \rangle|=n\}\;,\quad \mathcal{W}_n^k:=\mathcal{W}_n\cap\mathcal{W}(\mathcal{I}_{\leq k})\;, \forall k\geq0
\]
of length $n$.

These give a natural structure of a positively graded set on $\mathcal{W}(\mathcal{I})$ and on $\mathcal{W}(\mathcal{I}_{\leq k}
)$, for each $k\geq0$.

\subsection{Rings}

\subsubsection{Polynomial rings}

For a set $A$, we will write $\mathscr{Y}(A)$ for the \textit{multiplicative} free abelian group whose generator set is $A$. Thus, a formal group isomorphism $\exp: \Z^{A}\to \mathscr{Y}(A)$ is in place, giving a bijection $a\in A \mapsto \exp(a)\in \mathscr{Y}(A)$ on generators.

Setting $\mathscr{Y}_+(A) :=\exp(\Z_{\geq0}^A)< \mathscr{Y}(A)$, allows us to consider 
\[
\Z[A]:=\Z^{\mathscr{Y}_+(A)}
\]
as a polynomial ring in the formal variables $\{\exp(a)\}_{a\in A}$\footnote{We note that the slightly less formal discussion in the introduction section omitted the $\exp$ notation for ease of narrative.}, while 
\[
\Z[A^{\pm1}]:=\Z^{\mathscr{Y}(A)}
\]
being the larger ring of Laurent polynomials in same variables.

We write
\[
\theta = \sum_{\m\in \Z^{A}} c_\m \exp(\m)\in \Z[A^{\pm1}]\quad \mapsto\quad \theta_+:=\sum_{\m\in \Z_{\geq0}^{A}} c_\m \exp(\m)\in\Z[A]
\]
for the natural truncation map (which is not a ring homomorphism).

We similarly define a group homomorphism $\theta\mapsto \theta_+$ from $\mathrm{Maps}(\mathscr{Y}(A),\Z)$ into $\mathrm{Maps}(\mathscr{Y}_+(A),\Z)$, when dealing with functions of infinite support.

\subsubsection{Free rings}

The group 
\[
\mathscr{Fr}(A):=\Z^{\mathcal{W}(A)}
\]
has a natural structure of a free non-commutative associative ring, in the variables given by $A$. Similarly, we take note of the \textit{quantized} $\mathbb{Q}(q)$-algebra  
\[
\mathscr{Fr}_q(A):=\mathbb{Q}(q)^{\mathcal{W}(A)}\;,
\]
where $\mathbb{Q}(q)$ is the field of formal rational functions.

We shortcut notation to $R:= \mathscr{Fr}(\mathcal{I})$, $R(q):= \mathscr{Fr}_q(\mathcal{I})$ and write their subrings $R\{k\}:=\mathscr{Fr}(\mathcal{I}_{\leq k}) $, $R(q)\{k\}:=\mathscr{Fr}_q(\mathcal{I}_{\leq k}) $, for each $k\geq0$.

When setting 
\[
R_n:= \Z^{\mathcal{W}_n}\,,\; R_n(q):= \mathbb{Q}(q)^{\mathcal{W}_n}\,,\; R_n\{k\}:=\Z^{\mathcal{W}^k_n}\,,\;R_n(q)\{k\}:=\mathbb{Q}(q)^{\mathcal{W}^k_n}\;,
\]
we see gradings on the ring $R$ and the $\mathbb{Q}(q)$-algebra $R(q)$.

We may naturally view $R = \varinjlim_k R\{k\}$ and $R(q) = \varinjlim_k R(q)\{k\}$ as direct limits of graded (finitely generated) rings.

\subsection{Multisegments}

The formal notion of \textit{segments} will consists of ordered pairs of integers
\[
    \seg = \left\{ [a,b] \mid a, b \in \mathbb{Z},\; a\leq b \right\}\;.
\]
We refer to elements of $\Z_{\geq0}^{\seg}$ as \textit{multisegments}, and to elements of $\mathcal{W}(\seg)$ as \textit{ordered multisegments}.

We say that $\underline{\m}\in \mathcal{W}(\seg)$ is an \textit{ordering} of a multisegment $\m\in \Z_{\geq0}^{\seg}$, when $\langle \underline{\m} \rangle = \m$.

For instance, $\underline{\m}=[1,3][0,1][1,3]$ is a possible ordering of the multisegment $\m = [0,1] + 2[1,3]$.

\subsubsection{Integral decomposition of rings}
For each number $\eta\in \C^\times$, we see a copy of $\seg$ embedded into $\widetilde{\seg}$, the set of segments as defined in the introduction section, in the following manner:
\[
\seg(\eta):= \{ [\eta v^{2a}, \eta v^{2b}]\in \widetilde{\seg}\}_{[a,b]\in \seg}\;.
\]
The disjoint decomposition
\[
\widetilde{\seg} = \bigsqcup_{\eta\in \C^\times/\langle v^2\rangle} \seg(\eta) 
\]
according to multiplicative cosets of the group generated by $v^2$, induces a natural decomposition 
\begin{equation}\label{eq:decomp-ring}
\Z\left[\widetilde{\seg}\right] = \bigotimes_{\eta\in \C^\times/\langle v^2\rangle} \Z[\seg(\eta)]
\end{equation}
of polynomial rings. We write $\alpha_\eta:\Z[\seg] \xrightarrow{\sim} \Z[\seg(\eta)]< \Z\left[\widetilde{\seg}\right]$ for the resulting ring embeddings.

\subsubsection{Gradings}
When setting
\[
\seg_i =\{[a,b]\in \seg\,:\, b-a+1= i\}\;,\quad \seg^i =\{[a,b]\in \seg\,:\, b= i\}\;, 
\]
we see the graded structures $\seg = \bigsqcup_{i\geq1} \seg_i = \bigsqcup_{i\in \Z} \seg^i$.



We write $\m = \sum_{i\in \Z} \m[i]\in \Z^{\Seg}$, with $\m[i]\in \Z^{\Seg^i}$, according to the latter grading.

For $\m\in \Z^{\seg}$, we write $E(\m)= \{i\in \Z\,:\, \m[i]\neq0\}$ for the finite set of end points of segments in $\m$.

We say that a multisegment $\m \in \Z_{\geq0}^{\Seg}$ is \textit{right-aligned}, when $|E(\m)|\leq1$, that is, when $\m = \m [i_0]$ holds, for a single $i_0\in \Z$.

We say that an ordering $\underline{\m}\in \mathcal{W}(\seg)$ of a multisegment $\m\in \Z_{\geq0}^{\seg}$ is \textit{admissible}, when 
\[
\underline{\m} = \underline{\m[d_1]}\cdot \ldots \cdot \underline{\m[d_r]}
\]
holds, for some orderings $\underline{\m[d_i]}\in \mathcal{W}(\seg)$ of $\m[d_i]$, with $E(\m) = \{d_1<\ldots<d_r\}$.

For the multisegment \(\m=[0,2]+[2,2]+2[1,3]\), we have
\(\m[2]=[0,2]+[2,2]\), \(\m[3]=2[1,3]\), and \(E(\m)=\{2,3\}\).
For instance, \([2,2][0,2][1,3][1,3]\) is an admissible ordering of \(\m\), whereas \([1,3][0,2][2,2][1,3]\) is not.

\subsubsection{Support functions}

We denote the maps
\[
\supp,\, \delta,\, \epsilon: \seg \to \Z_{\geq0}^{\mathcal{I}}\,,
\]
\[
\;\supp([a,b]) = \alpha_{a} + \alpha_{a+1} + \ldots + \alpha_{b}\;,\; \delta([a,b])= \alpha_a\;,\; \epsilon([a,b])=\alpha_b
\]
and extend it to group homomorphisms $\supp,\delta,\epsilon: \Z^{\seg} \to \Z^{\mathcal{I}}$, all of which clearly map the monoid $\Z_{\geq0}^{\seg}$ into $\Z_{\geq0}^{\mathcal{I}}$.

We write \[
\seg\{k\} := \{[a,b]\in \seg\,:\,-k\leq a, b\leq k\} = \supp^{-1}(\Z^{\mathcal{I}_{\leq k}})\cap\seg\;,
\]
for $k\geq0$.

Let us take note of the positively graded set structures on the monoids of multisegments 
\[
\Z_{\geq0}^{\seg} = \bigsqcup_{n=0}^\infty (\Z_{\geq0}^{\seg})_n\;,\quad\Z_{\geq0}^{\seg\{k\}} = \bigsqcup_{n=0}^\infty (\Z_{\geq0}^{\seg\{k\}})_n\;, k\geq0\;,
\]
that arise when defining $(\Z_{\geq0}^{\Seg})_n:= \{\m\in \Z_{\geq0}^\seg\,:\, | \supp(\m) | = n\}$ and $(\Z_{\geq0}^{\Seg\{k\}})_n= (\Z_{\geq0}^{\Seg})_n\cap \Z_{\geq0}^{\Seg\{k\}}$. 

We may write a grading
\[
\Z[\seg] = \bigoplus_{n=0}^\infty \Z[\seg]_n
\]
on the polynomial ring, by setting $\Z[\seg]_n:=\Z^{\exp((\Z_{\geq0}^{\seg})_n)}$.

In these terms, we have
\[
E(\m) = \{d\in \Z\::\: \epsilon(\m)(d)\neq0\}\;,
\]
for $\m\in \Z_{\geq0}^{\seg}$, and we similarly write
\[
D(\m) = \{d\in \Z\::\: \delta(\m)(d)\neq0\}
\]
for the set of begin points of segments in $\m$.

\begin{lemma}\label{lem:monotone}
    For any $\m\in \Z^{\seg}$ and $i\in \Z$, we have
    \[
    \supp(\m)(i + 1) - \supp(\m)(i) = \delta(\m)(i+1) - \epsilon(\m)(i)\;.
    \]
\end{lemma}

\begin{proof}
    By additivity it is enough to assume that $\m = \Delta\in \seg$. In this case, both sided of the equality are easily computed.
\end{proof}

\subsubsection{Spherical multisegments}

We say that a segment $[a,b]\in \seg$ \textit{precedes} another segment $[c,d]\in \seg$, when $a<c\leq b+1< d+1$ holds.

We say that a multisegment $\m\in \Z_{\geq0}^{\seg}$ is \textit{spherical}, when in any, or one, ordering $\underline{\m}= \Delta_1\cdot\ldots \cdot \Delta_r\in \mathcal{W}(\seg)$ of $\m$, $\Delta_i$ does not precede $\Delta_j$, for any $1\leq i,j\leq r$.

Clearly, right-aligned multisegments are spherical.

It is easy to verify that for any $\beta\in \Z_{\geq0}^{\mathcal{I}}$, there is a unique spherical multisegment $\m_\beta\in \Z_{\geq0}^{\seg}$ with $\supp(\m_\beta)= \beta$.

For any multisegment $\m\in \Z_{\geq0}^{\seg}$, we denote $\m^\circ\in \Z_{\geq0}^{\seg}$ to be the spherical multisegment with $\supp(\m^\circ) = \supp(\m)$.

For example, when $\m = [0,1] + [2,2] + 2[1,3]$, we have $\supp(\m) = \alpha_0 + 3\alpha_1 + 3\alpha_2 + 2\alpha_3$ and $\m^\circ = [0,3] + [1,3]+ [1,2]$.

We will need the following criterion for detecting whether $\m^\circ$ is right-aligned.

\begin{proposition}\label{prop:crit-right}
     For a multisegment $\m\in \Z_{\geq0}^{\seg}$ and $b\in \Z$, consider the shifted multiset
        \[
        \epsilon'_b(\m) = \sum_{b\neq i\in \Z} \epsilon(\m)(i) \,\alpha_{i+1}\in \Z_{\geq0}^{\mathcal{I}}\;,
        \]

Then, the spherical multisegment $\m^\circ$ is right-aligned, if and only if, $b\in \Z$ exists, for which $\delta(\m) - \epsilon'_b(\m)\in \Z_{\geq0}^{\mathcal{I}}$.

When the above condition holds, we have
\[
\m^\circ = \m^\circ[b] = \sum_{a\in \Z} \left(\delta(\m)(a) - \epsilon'_b(\m)(a)\right)  [a,b]\;,
\]
and $\delta(\m^\circ) = \delta(\m) - \epsilon'_b(\m)$.
\end{proposition}
\begin{proof}

For a multisegment $\m\in  \Z_{\geq0}^{\seg}$ and $\max E(\m)\leq b\in \Z$, let us consider 
\[
\m_b = \sum_{a\in \Z} \left(\delta(\m)(a) - \epsilon'_{b}(\m)(a)\right)  [a,b]\in \Z^{\seg}\;.
\]
Let us note that $\supp(\m) = \supp(\m_b)$ holds. 

Indeed, when $\m = [a,c] \in \seg$ is a segment with $c\leq b $, we have $\m_b = [a,b] - [c+1,b]$ (with $[b+1,b]= 0$), and the equality is visible. The general case now follows by additivity.

Suppose now that $0\neq \m\in \Z_{\geq0}^{\seg}$ is given with $\delta(\m) - \epsilon'_{b_0}(\m)\in \Z_{\geq0}^{\mathcal{I}}$.

It is evident that $\delta(\m)-\epsilon'_b(\m)$ cannot belong to $\Z_{\geq0}^{\mathcal{I}}$ when $b\neq \max E(\m)$. Thus $b_0 = \max E(\m)$.

Then, $\m_{b_0}\in \Z_{\geq0}^{\seg}$ becomes a right-aligned multisegment with $\supp (\m_{b_0}) = \supp(\m)$. Since right-aligned mutlisegments are spherical, we have $\m_{b_0} = \m^\circ$.

Conversely, suppose that $\m\in \Z_{\geq0}^{\seg}$ has $\m^\circ = \m^\circ[b]$. That implies $\supp(\m^\circ)(i)\leq \supp(\m^\circ)(i+1)$, and consequently, $\supp(\m)(i)\leq \supp(\m)(i+1)$, for all $b\neq i\in \Z$. 

The last condition now implies $\delta(\m) - \epsilon'_{b}(\m)\in \Z_{\geq0}^{\mathcal{I}}$ by Lemma \ref{lem:monotone}.

\end{proof}

\section{Representation theory of type $A$ affine Hecke algebras}

For a complex associative algebra $H$, we denote its abelian category of finite-dimensional modules as $\Rep(H)$ and often treat it as its collection of objects.

We write $\Irr(H)$ for the set of isomorphism classed of simple modules in $\Rep(H)$. 

We denote $G(H)= \Z^{\Irr(H)}$ to be the Grothendieck group of $\Rep(H)$, and write $[V]\in G(H)$ for the semisimplified class of a module $V\in \Rep(H)$.

For a module $V\in \Rep(H)$, writing $[V] = \sum_{L\in \Irr(H)} c_{V,L} [L]$, we take note of the multiplicities 
\[[V:L]:= c_{V,L}
\]
of simple modules in the Jordan--H\"older series of $V$. We say that $L\in \Irr(H)$ \textit{appears} in $[V]$, when $L$ is isomorphic to a subquotient of the module $V$, or equivalently, $[V:L]\neq0$.

\subsection{Hecke algebras}

Let $v\in \mathbb{C}$ be a fixed non-root of unity.

Given $n\in \mathbb{Z}_{>0}$, the $GL_n$ root datum gives rise to the (extended) \textit{affine Hecke algebra} $H_{n}= H_{n,v}$: This is the complex associative algebra generated by $T_1,\ldots, T_{n-1}$ and invertible $y_1,\ldots,y_n$, subject to the relations
\[
\begin{array}{ll}
T_i T_{i+1} T_i = T_{i+1} T_i T_{i+1},\; & \forall 1\leq i\leq n-2\\

(T_i -v^2)(T_i+1)=0,\;& \forall 1\leq i \leq n-1\\

T_iT_j = T_j T_i,\;  & \forall |j-i|>1\\

y_iy_j = y_jy_i,\;& \forall 1\leq i,j\leq n\\

T_i y_iT_i = v^2 y_{i+1},\; &\forall 1\leq i\leq n-1\\

T_i y_j = y_jT_i,\; &\forall j\neq i, i+1\;.
\end{array}\;
\]

We often treat $H_0=\C$ as the trivial algebra.

\subsubsection{Restriction and induction}
For a composition $\underline{n} = (n_1, \ldots ,n_r)$ of $n$, the algebra $H_{\underline{n}}: =  H_{n_1}\otimes \cdots \otimes H_{n_r}$ is naturally embedded as
\[
\iota_{\underline{n}}: H_{\underline{n}} \to H_n\;,
\]
by sending the generators $\tilde{T}_i,\tilde{y}_i$ of $H(n_j,q)$ to $T_{n_1+ \ldots + n_{j-1} + i}, y_{n_1+ \ldots + n_{j-1} + i}$ in $H_n$.

This embedding gives rise to a restriction functor
\[
\rres_{\underline{n}}: \Rep(H_n) \to \Rep(H_{\underline{n}})\;.
\]
The exact induction functor
\[
\iind_{\underline{n}}: \Rep(H_{\underline{n}}) \;\to\; \Rep(H_n)\;,\quad
\iind_{\underline{n}}(V) = H_n\otimes_{\iota_{\underline{n}}(H_{\underline{n}})}V
\]
is left-adjoint to $\rres_{\underline{n}}$.

Given modules $V_i\in \Rep(H_{n_i})$, for $i=1,\ldots,r$, we denote
\[
V_1\times\cdots \times V_r := \iind_{\underline{n}}(V_1\boxtimes \ldots \boxtimes V_r)\in \Rep(H_n)\;.
\]
\begin{remark}\label{rem:ss-order}
It is known (e.g. \cite[Theorem 1.9]{zel-main}) that an equality
\[
[V_1\times\cdots \times V_r] = [V_{\sigma(1)}\times\cdots \times V_{\sigma(r)}]
\]
holds in $G_1(H_n)$, for any permutation $\sigma$ of the index set $\{1,\ldots,r\}$.
\end{remark}

\subsubsection{Weight decomposition and central characters}
The composition $(1,\ldots,1)$ of $n$ gives a subalgebra $P_n =\mathbb{C}[y_1^{ \pm 1},\ldots y_n^{ \pm 1}]=\iota_{(1,\ldots,1)}(H_{(1,\ldots,1)})< H_n$ of Laurent polynomials.

Restricting to $P_n$ provides a weight space decomposition for each module $V\in \Rep(H_n)$ of the form
\[
V= \oplus_{\chi\in (\mathbb{C}^\times)^n} V_\chi\;,
\]
where $V_\chi$, for $\chi = (\chi_1,\ldots,\chi_n)$, denotes the common generalized eigenspace of $P_n$ in $M$, given by the character $y_i \mapsto \chi_i$.

Recall that the center subalgebra of $H_n$ is given by the symmetric Laurent polynomials in $P_n$ (\cite[Proposition 3.11]{lusz-affine}). In particular, the complex central characters of $H_n$ are given by orbits of the action of $S_n$ (symmteric group) on $(\mathbb{C}^{\times})^n$.

Hence, the full subcategory $\Rep_{\overline{\chi}}(H_n)$, for a given $\overline{\chi} \in (\mathbb{C}^\times)^n/S_n$, of modules $V\in \Rep(H_n)$ that decompose as $V = \oplus_{ \chi\in \overline{\chi}} V_\chi$, is a Serre subcategory.




\subsubsection{Integral blocks}\label{sect:integral-block}

For $\eta\in \C^\times$, we denote by $\Rep_{\eta}(H_n)$ the Serre subcategory of $\Rep(H_n)$ consisting of modules in which $y_i$'s all act with eigenvalues of the form $\eta v^{2i}$, with $i\in \Z$. In other words,
\[
\Rep_{\eta}(H_n) = \oplus_{\overline{\chi}\in(\eta v^{2\mathbb{Z}})^n/S_n} \Rep_{\overline{\chi}}(H_n)\;.
\]
We take the convention of $\Rep_\eta(H_0) = \Rep(H_0)$.

We also write $\Irr_{\eta}(n)$ for the set of classes of simple modules in $\Rep_{\eta}(H_n)$, and 
\[
G_{\eta}(H_n) = \Z^{\Irr_{\eta}(n)}< G(H_n)
\]
for the corresponding Grothendieck group of $\Rep_{\eta}(H_n)$.

We similarly write $\Rep_{\eta}(H_{\underline{n}})$ and $G_{\eta}(H_{\underline{n}})$, for the analogous construction with a composition $\underline{n}$.

There is an evident family of algebra automorphisms $\sigma_{\eta}:H_n\to H_n$, for $\eta\in \C^\times$, that is given by $\sigma_{\eta}(T_i) = T_i$ and $\sigma_{\eta}(y_j) = \eta y_j$, on generators. Pulling back through it, provides equivalences of abelian categories
\begin{equation}\label{eq:shift-eta}
\sigma_{\eta}^\ast: \Rep_1(H_n)\overset{\sim}{\rightarrow}  \Rep_{\eta}(H_n)\;.
\end{equation}

Thus, the full category $\Rep(H_n)$ may naturally be viewed as a direct sum of a continuum of copies the integral block $\Rep_1(H_n)$, indexed by the cosets of $\C^\times/\langle v^2\rangle$.

The assignment
\[
\chi^{\eta}_w =(\eta v^{2i_1}, \ldots, \eta v^{2i_n})\in (\eta v^{2\mathbb{Z}})^n \;\mapsto\; w=\alpha_{i_1}\cdot\ldots\cdot\alpha_{i_n}\in \mathcal{W}_n
\]
sets a natural bijection between the $P_n$-weights appearing in a module $V\in \Rep_{\eta}(H_n)$ and the set $\mathcal{W}_n$ of words of length $n$ with letters in $\mathcal{I}$. We also write $\chi^1_{w} = \chi_w$.

The projection $( \eta v^{2\mathbb{Z}})^n \to (\eta v^{2\mathbb{Z}})^n/S_n$ onto orbits corresponds under this bijection to the map $w \mapsto \langle w\rangle$. Thus, we may identify the central characters that are admitted by modules in $\Rep_{\eta}(H_n)$ with elements of $\Z_{\geq0}^{\mathcal{I}}$ of height $n$.

In this sense, for a simple module $L\in \Irr_{\eta}(n)$, we write $\overline{\chi^{\eta}}_L\in \Z_{\geq0}^{\mathcal{I}}$ for its central character.
\begin{remark}\label{rem:chars-cent}
    For any $L_i \in \Irr_{\eta}(n_i)$, $i=1,2$, we have $L_1\times L_2\in \Rep_{\overline{\chi^{\eta}}_{L_1}+\overline{\chi^{\eta}}_{L_2}}(H_{n_1+n_2})$. In particular, any simple subquotient $L$ of $L_1\times L_2$ admits a central character $\overline{\chi^{\eta}}_L = \overline{\chi^{\eta}}_{L_1} + \overline{\chi^{\eta}}_{L_2}$.

Dually, whenever $L_1\boxtimes L_2$ is a simple subquotient of a module of the form $\rres_{(n_1,n_2)} (V)$, for $V\in \Rep_{\overline{\chi^{\eta}}}(H_{n_1+n_2})$, we must have $\overline{\chi^{\eta}}_{L_1}+\overline{\chi^{\eta}}_{L_2} = \overline{\chi^{\eta}}$.

\end{remark}


In particular, the group
\[
\mathbf{R_{\eta}}:= \bigoplus_{n=0}^\infty G_{\eta}(H_n)
\]
is equipped with a ring structure that arises from the induction product. It is a commutative ring by Remark \ref{rem:ss-order}.

The decomposition described in the following proposition is a well-known property.

\begin{proposition}\label{prop:bzring-decomp}
The parabolic induction maps
\[
([V_1], \ldots, [V_r])\in G_{\eta_1}(H_{n_1})\times \cdots\times G_{\eta_r}(H_{n_r})\;\mapsto\; [\iind(V_1\boxtimes\cdots \boxtimes V_r)]\in G(H_{n_1+\ldots + n_r})\;,
\]
for $\eta_1,\ldots,\eta_r\in \C^\times$ with $\eta_i/\eta_j\notin\{v^{2k}\}_{k\in \Z}$, for all $i,j$, extend into a ring isomorphism
\[
\bigotimes_{\eta\in \C^\times/\langle v^2\rangle} \mathbf{R}_{\eta} \;\cong \;\mathbf{R}:= \bigoplus_{n=0}^\infty G(H_n)\;,
\]
which preserves irreducibility.

Moreover, for all $\eta\in \C^\times$, the equivalences in \eqref{eq:shift-eta} combine into a ring isomorphism 
\[
\sigma_{\eta}^\ast: \mathbf{R}_1 \overset{\sim}{\rightarrow}  \mathbf{R}_{\eta}\;.
\]
Thus, the full Grothendieck ring $\mathbf{R}$ has a canonical decomposition as a tensor product of a continuum of rings, each of which is isomorphic with the integral Grothendieck ring $\mathbf{R}_1$.

\end{proposition}

In light of that decomposition, much of our discussion will be reduced to dealing with the integral blocks $\Rep_1(H_n)$, and their Grothendieck ring $\mathbf{R}_1$.

\subsubsection{Segment modules}
For a segment $\Delta=[a,b]\in \seg$, we write the $1$-dimensional module $Z(\Delta)\in \Irr_1(b-a+1)$, given on generators by $T_i \mapsto -1$ and $y_i\mapsto v^{2(b-i+1)}$.

For an ordered multisegment $\underline{\m} = \Delta_1\cdot\ldots\cdot \Delta_r\in \mathcal{W}(\seg)$, we now define the module
\[
M(\underline{\m}):= Z(\Delta_1)\times\cdots\times Z(\Delta_r)\in \Rep_1(H_n)\;,
\]
where $\langle \underline{\m}\rangle\in (\Z_{\geq0}^{\seg})_n$.

For a multisegment $\m\in (\Z_{\geq0}^{\seg})_n$, we may choose an ordering $\underline{\m}$ for it and define the semisimplified class of representations
\[
M(\m):= [ M(\underline{\m})]\in G_1(H_n)\;.
\]
This construction is well-defined, that is, independent of choice of ordering, due to Remark \ref{rem:ss-order}.

We can furthermore define a ring homomorphism by
\begin{equation}\label{eq:zel-map}
 \underline{\zeta} :\; \Z[\seg]\; \overset{\sim}{\to}\; \mathbf{R_1}\;,\quad \underline{\zeta}(\exp(\m)) = M(\m) \;.
\end{equation}

\subsubsection{Zelevinsky classification}\label{sect:zel}
For $n\geq0$ and an admissible ordering $\underline{\m} \in \mathcal{W}(\seg)$ of a multisegment $\m\in (\Z_{\geq0}^{\seg})_n$ we set $\zeta(\m):= M(\underline{\m})$. It is known that the isomorphism class of the representation is independent of a choice of an admissible ordering.

Modules of the form $\zeta(\m)\in \Rep_1(H_n)$, for $\m\in (\Z_{\geq0}^{\seg})_n$, are the \textit{standard}\footnote{In the context of $p$-adic groups, this class is often named `Zelevinsky-standard'. } representations.

Those modules, while possibly reducible, admit a unique simple quotient module $\zeta(\m)\to Z(\m)$.
\begin{proposition}

\begin{enumerate}
    \item (\cite[Theorem 2.6]{soll-survey}\cite[Theorem 6.5]{zel-main})
The resulting map
\[
\m\in \Z_{\geq0}^{\seg} \;\mapsto\; Z(\m)\in \Irr_1:= \bigsqcup_{n\geq0}\Irr_1(n)
\]
is a bijection.

\item (\cite[Corollary 7.5]{zel-main})
The map $\underline{\zeta}$ is a ring isomorphism. In particular, the ring $\mathbf{R_1} = \Z^{\Irr_1}$ is polynomial.

\end{enumerate}
\end{proposition}

The isomorphism $\underline{\zeta}$ is easily extendable to an isomorphism 
\[
\widetilde{\underline{\zeta}}: \Z\left[\widetilde{\seg}\right] \xrightarrow{\sim} \mathbf{R}
\]
of rings, by decomposing both sides according to \eqref{eq:decomp-ring} and Proposition \ref{prop:bzring-decomp}, and setting
\begin{equation}\label{eq:intertwHeta}
 \sigma_\eta^\ast \circ \underline{\zeta} \circ\alpha_\eta^{-1} :\Z[\seg(\eta)] \xrightarrow{\sim} \mathbf{R}_\eta
\end{equation}
on each tensor factor.


\subsection{Mackey theory}

The semisimplified modules $[\rres_{\underline{n}}(M(\m))]\in G_1(H_{\underline{n}})$ that arise from restriction of segments modules have a well-known description in terms of the appropriate version of Mackey theory. We will now set the notations for that description.

For a segment $[x,y]\in \seg$ and an integer $s\geq 1$, let us consider the set of words
\[
S'(x,y,s) = \left\{
(a_{j}, e_{j})_{j=1}^{k}\in \mathcal{W}(\Z\times\Z) \,\middle|\,
\begin{array}{l}
x = a_{k} < \cdots < a_1 \leq y, \\
1\leq  e_1 < \cdots < e_{k} \leq s
\end{array}
\right\}.
\]

For an ordered multisegment $\underline{\m} = [x_1,y_1]\cdots[x_r,y_r]\in \mathcal{W}(\seg)$, we consider the set of bitableaux
\[
\mathcal{K}(\underline{\m},s) = \{\mathcal{Q}_1\cdots \mathcal{Q}_r\in \tab(\Z\times \Z) \,:\, \mathcal{Q}_i\in S'(x_i,y_i,s),\,\forall\,1\leq i\leq r\}\;.
\]

Given an bitableau $\mathcal{Q} = (a^i_j,e^i_j)_{(i,j)\in I_{\mathcal{Q}}}\in \KK $, let us write a decomposition
\[
I_{\mathcal{Q}} = I_{\mathcal{Q}}^1 \sqcup\cdots \sqcup I_{\mathcal{Q}}^s\;,
\]
so that $e^i_j =e$, for all $1\leq e\leq s$ and $(i,j)\in I_{\mathcal{Q}}^e$.

Subsequently, we define the multisegment 
\[
\m_{\mathcal{Q},e}:= \sum_{(i,j)\in I_{\mathcal{Q}}^e} [a^i_j, a^i_{j-1}-1]\in \Z_{\geq0}^{\seg}\;,
\]
with $a^i_0=y_i+1$ being assumed.

In these terms we denote the semisimplified class of representations
\[
M(\mathcal{Q}):= M\left(\m_{\mathcal{Q},1}\right)\boxtimes \cdots\boxtimes M\left(\m_{\mathcal{Q},s}\right)\in G_1(H_{\underline{n}(\mathcal{Q})})\;,
\]

for all $\mathcal{Q}\in \KK$. Here, $\underline{n}(\mathcal{Q})=(n_{\m_{\mathcal{Q},1}},\ldots, n_{\m_{\mathcal{Q},s}})$.

\begin{proposition}\label{prop:mackey}

Let $\underline{\m}$ be any ordering of a given multisegment $\m \in (\Z_{\geq0}^{\seg})_n$.

When considering the set of compositions 
\[
C(n, s) = \left\{(n_1,\ldots,n_s)\in (\mathbb{Z}_{\geq0})^{s}:\;\sum_{i=1}^s n_i=n\right\}\;,
\]
we have an equality
    \[
\sum_{\underline{n}\in C(n,s)}\rres_{H_{\underline{n}}}
\bigl(M(\m)\bigr)= \sum_{\mathcal{Q}\in \KK} M(\mathcal{Q})
\]

    in the corresponding sum of Grothrendieck groups $\bigoplus_{\underline{n}\in C(n,s)} G_1(H_{\underline{n}})$. 
\end{proposition}

\begin{proof}
For a segment $[x,y]\in \seg_n$ and $n_1+\ldots+ n_s=n$, the restricted module 
\[
\rres_{(n_1,\ldots,n_s)}(Z([x,y]))= L_1\boxtimes \cdots \boxtimes L_s
\]
is $1$-dimensional, and evidently, $L_j\cong Z([a_j, a_{j-1}-1])$, for all $1\leq j\leq s$, where $a_j = y+1- (n_1+\ldots +n_{j})$.

The statement now follows from a standard semisimplified version of Mackey theory in this setting, as explicated in \cite[Section 2.2]{lapid-minguez-parab}.

\end{proof}


\subsection{Characters}\label{sect:ch}

A homomorphism $\chi_w:H_{(1,\ldots,1)}\cong P_n=\mathbb{C}[y_1^{\pm1},\ldots,y_n^{\pm1}]\to \mathbb{C}$, as constructed in Section \ref{sect:integral-block}, for a given word $w\in \mathcal{W}_n$, may be viewed as an element of $\Irr_1(H_{(1,\ldots,1)})$. Thus, we may write a group isomorphism
\[
\bar{\tau}:=\sum_{w\in \mathcal{W}_n} c_ww\in R_n= \Z^{\mathcal{W}_n}\quad\leftrightarrow\quad \tau=\sum_{w\in \mathcal{W}_n} c_w[\chi_w]\in G_1(H_{(1,\ldots,1)})\;.
\]
The resulting additive map that is given by
\[
\mathrm{ch}: G_1(H_n)\to R_n\,,\quad \mathrm{ch}([V]) = \overline{[\rres_{(1,\ldots,1)}(V)]}\;,\; V\in \Rep_1(H_n)\;,
\]
is the \textit{character} map for affine Hecke algebra modules.

It is known (e.g. \cite[Theorem 5.3.1]{klsh-book})  to be an injective map, which we rewrite explicitly as 
\[
\mathrm{ch}([V]) = \sum_{w\in \mathcal{W}_n} \dim(V_{\chi_w})\,w\;,
\]
for any module $V\in \Rep_1(H_n)$.

For convenience, we will denote the multiplicities in the last equation as \[
\dim(V_{\chi_w}) = [V:w]\;.
\]

It follows from Remark \ref{rem:chars-cent}, that whenever a module $V\in \Rep_{\overline{\chi}}(H_n)$ has $[V:w]\neq0$, for a word $w\in \mathcal{W}_n$, we must have $\langle w \rangle = \overline{\chi}$.

For example, the $1$-dimensional segment $H_2$-module $Z([0,1])$ has
\(\ch(Z([0,1]))=\alpha_0\alpha_1\), while the standard $H_2$-module $M=Z([0,0])\times Z([1,1])$ is $2$-dimensional and has
\(\ch(M)=\alpha_0\alpha_1+\alpha_1\alpha_0\).

\subsubsection{Restriction}

It is clear from definitions, that for any composition $(n_1,\ldots,n_r)$ of $n$, a formula

\begin{equation}
    \mathrm{ch}([V]) = \sum_{L_1,\ldots, L_r} [\rres_{(n_1,\ldots,n_r)} (V): L_1\boxtimes \cdots \boxtimes L_r] \,\mathrm{ch}(L_1)\cdot\ldots\cdot \mathrm{ch}(L_r)\in R_n
\end{equation}
holds in the free ring $R$, for any module $V\in \Rep_1(H_n)$, where the summation is over tuples of simple modules $L_i \in \Irr_1(n_i)$, $i=1,\ldots,r$.

In particular, for any word $w\in \mathcal{W}_n$, we may write
\begin{equation}\label{eq:word-concat}
[V:w]=  \sum_{L_1,\ldots, L_r} [\rres_{(n_1,\ldots,n_r)} (V): L_1\boxtimes \cdots \boxtimes L_r] [L_1:w_1]\cdot\ldots \cdot [L_r:w_r]\;,
\end{equation}
where $w = w_1\cdot\ldots\cdot w_r$ is the unique decomposition with $w_i\in \mathcal{W}_{n_i}$.

\subsubsection{Spherical modules and descending words}

Let us recall some known properties that are exhibited by spherical multisegments in the Zelevinsky classification.

\begin{lemma}\label{lemma:m_circ}
\begin{enumerate}

\item\label{it:mc1} For a spherical multisegment $\m = \m^\circ\in \Z_{\geq0}^{\seg}$, the module $M(\m)$ is simple. In particular, $[\zeta(\m)]= M(\m)= Z(\m)$.
    
\item\label{it:mc3} For any multisegment $\m\in \Z_{\geq0}^{\seg}$, the multiplicity $[M(\m):Z(\m^\circ)]$ equals $1$.

\item\label{it:mc33} For any multisegment $\m\in \Z_{\geq0}^{\seg}$, the standard representation $\zeta(\m)$ has a unique simple submodule, which is isomorphic to $Z(\m^\circ)$.
    
\end{enumerate}
\end{lemma}

\begin{proof}
Using the standard categorical equivalences that originate in \cite{Bo} (see, for example, \cite[Section 3.2.2]{gur-jlms} for a detailed account), we may transfer the problem to the category of smooth representations of $\mathrm{GL}_n(F)$, where $F$ is a $p$-adic field. In this setting, the statements are well known, though often not stated explicitly in a single place.

A convenient reference is provided by \cite{ming-secher}: part~\eqref{it:mc1} follows from \cite[Theorem 3.10]{ming-secher}, while part~\eqref{it:mc3} is a consequence of \cite[Proposition 4.4]{ming-secher} together with \cite[Lemma 4.1]{ming-secher}.

Finally, part~\eqref{it:mc33} follows from the arguments in \cite[Proposition 2.3]{ky-invent}.
\end{proof}

We say that a word $w = \alpha_{i_1}\cdot\ldots\alpha_{i_k}\in \mathcal{W}(\mathcal{I})$ is \textit{descending}, when $i_1\geq  \ldots \geq i_k$.

It is easy to see that for a multisegment $\m\in (\Z_{\geq0}^{\seg})_n$, there is a unique descending word $\pi(\m)\in \mathcal{W}_n$ with $\langle \pi(\m)\rangle = \supp(\m)$. It is also clear that $\pi(\m) = \pi(\m^\circ)$, and we may denote 
\[
w_\beta = \pi(\m)\;,
\]
by choosing any $\m\in \Z_{\geq0}^{\seg}$ with $\supp(\m)=\beta\in \Z_{\geq0}^{\mathcal{I}}$.

Let us denote the multiplicity $r_{\beta}:=[Z(\m_\beta):w_\beta]$, for each $\beta\in \Z_{\geq0}^{\mathcal{I}}$.

\begin{proposition}\label{prop:descending-spherical}
For any $\beta\in \Z_{\geq0}^{\mathcal{I}}$, $r_\beta>0$ holds, and for any module $V\in \Rep_1(H_{|\beta|})$, we have an equality $[V: w_{\beta}] = r_{\beta}\cdot [V:Z(\m_\beta)]$ of multiplicities.
\end{proposition}

\begin{proof}
By additivity, for the second statement it is enough to show that whenever $V$ is simple and $[V:w_{\beta}]$ non-zero, we must have $[V]= Z(\m_\beta)$.

Indeed, when $[V:w_{\beta}]$ is non-zero, $\chi_{w_{\beta}}$ must appear as a submodule of $\rres_{(1,\ldots,1)}(V)$. By functor adjunction, if $V$ is assumed to be simple, we see that $V$ is isomorphic with a quotient module of $\iind_{(1,\ldots,1)}(\chi_{w_{\beta}})$.

Now, a known symmetry of the induction product (e.g. \cite[Proposition 2.2]{gur-jems}) then implies that $V$ is isomorphic to a submodule of $\iind_{(1,\ldots,1)}(\chi_{w^\vee_{\beta}})$, where $w^\vee_{\beta}\in \mathcal{W}_n$ denotes the reading of the word $w_{\beta}$ in the reversed order.

Note, that when writing $w_\beta= \alpha_{i_1}\cdots \alpha_{i_n}$, we see that $\iind_{(1,\ldots,1)}(\chi_{w^\vee_{\beta}})= M(\underline{\m})$, for the ordered multisegment 
\[
\underline{\m} = [i_n,i_n]\cdots [i_1,i_1]\in \mathcal{W}(\seg)\;.
\]
Moreover, since $w_\beta$ was a descending word, $\underline{\m}$ becomes an admissible ordering of $\langle \underline{\m}\rangle\in \Z_{\geq0}^{\seg}$. 

Thus, $M(\underline{\m})$ is a standard representation. By Lemma \ref{lemma:m_circ}\eqref{it:mc33}, we obtain $[V]= Z(\langle \underline{\m}\rangle^\circ) = Z(\m_\beta)$.

Finally, for the first statement, note that all preceding arguments may be reversed to produce $\chi_{w_\beta}$ as a submodule of $\rres_{(1,\ldots,1)}(Z(\m_\beta))$, and to imply that $r_\beta\neq0$.

\end{proof}

\subsubsection{Indicator modules and words}\label{sect:indicators}

For a multisegment $\m\in (\Z_{\geq0}^{\seg})_n$ with $E(\m) = \{d_1<\ldots< d_{l}\}$, let us write $\underline{n_{\m}} = (|\beta^{\m}_1|,\ldots,|\beta^{\m}_l|)$ for the composition of $n= n_{\m}$, where $\beta^{\m}_i= \supp(\m[d_i])$, $i=1,\ldots,l$, and $l_{\m}=l$.

A moment's reflection would show that the tuple $(\beta^{\m}_1,\ldots,\beta^{\m}_{l_{\m}})$ in fact determines the multisegment $\m$ uniquely.

Let us denote the \textit{indicator word} 
\[
w(\m):= w_{\beta^{\m}_1}\cdots w_{\beta^{\m}_{l_{\m}}}\in \mathcal{W}_{n_{\m}}\;,
\]
and the integer $r_{\m}:=r_{\beta^{\m}_1}\cdots r_{\beta^{\m}_{l_{m}}}$.

We define
\[
\zeta(\m)^{\otimes}:= \zeta(\m[d_1])\boxtimes \cdots \boxtimes \zeta(\m[d_{l_{\m}}]) \in \Rep_1(H_{\underline{n_{\m}}})
\]
to be the \textit{indicator module} of $\m$.

By Lemma \ref{lemma:m_circ}\eqref{it:mc1}, indicator modules are simple. It also follow from the construction of standard modules that
\[
\zeta(\m) = \iind_{\underline{n_{\m}}}(\zeta(\m)^{\otimes})
\]
holds.

Given a simple module $L= Z(\m)\in \Irr_1$, let us denote $\zeta(L):= \zeta(\m)$, $\zeta(L)^\otimes:= \zeta(\m)^{\otimes}$, $w(L):=w(\m)$, $r_L:=r_{\m}$ and $\underline{n_L}:= \underline{n_{\m}}$.

\begin{proposition}\label{prop:char-word}
For any simple module $L\in \Irr_1(n)$ and any module $V\in \Rep_1(H_{n})$, we have an equality 
\[
[V: w(L)] = r_{L}\cdot [\rres_{\underline{n_L}} (V) : \zeta(L)^\otimes]
\]
of multiplicities.

\end{proposition}

\begin{proof}

Let us write $L=Z(\m)$, for a multisegment $\m\in \Z_{\geq0}^{\seg}$.

According to \eqref{eq:word-concat}, we have
\[
[V:w(L)]=  \sum_{L_1, \ldots, L_l} [\rres_{\underline{n_L}} (V): L_1\boxtimes \cdots \boxtimes L_r] [L_1:w_{\beta^{\m}_1}]\cdot\ldots \cdot [L_r:w_{\beta^{\m}_l}]\;.
\]
By Proposition \ref{prop:descending-spherical}, the right-hand side of the last equation equals 
\[
[\rres_{\underline{n_L}} (V): Z(\m_{\beta^{\m}_1})\boxtimes \cdots \boxtimes Z(\m_{\beta^{\m}_l})] r_{\beta^{\m}_1}\cdot\ldots\cdot r_{\beta^{\m}_l} \;.
\]

Recall that $\beta^{\m}_i= \supp(\m[d_i])$, for $i=1,\ldots,l$ with $E(\m)= \{d_1,\ldots,d_l\}$. Since $\m[d_i]$ are right-aligned, hence, spherical multisegments, we have $\m_{\beta^{\m}_i} = \m[d_i]$ and $Z(\m_{\beta^{\m}_i}) = \zeta(\m[d_i])$.

\end{proof}

\subsubsection{Reciprocal characters}

For a given $n\geq0$, let us take note of the \textit{reciprocal character} map
\[
\mathrm{ch}^\otimes: G_1(H_n) \to \Z[\seg]_n
\]
that is given by
\begin{equation}\label{eq:recip-def}
\mathrm{ch}^\otimes(V) = \sum_{\m\in (\Z_{\geq0}^{\seg})_n}[\rres_{\underline{n_{\m}}}(V)\,:\, \zeta(\m)^\otimes] \;\exp(\m)\;.
\end{equation}

\begin{remark}
    
By functor adjunction, since standard modules have unique simple quotients, we now have
\[
\dim \Hom_{H_{\underline{n_L}}}\left( \zeta(L)^\otimes, \rres_{\underline{n_L}}(L') \right) = \dim \Hom_{H_{n}}\left( \zeta(L), L') \right) = \left\{ \begin{array}{ll} 1 & L\cong L' \\ 0 & L\not\cong L' \end{array} \right.\;,
\]
for any $L,L'\in \Irr_1(n)$.

Yet, it is still possible for the indicator module $\zeta(L)^{\otimes}$ to appear as a subquotient of $\rres_{\underline{n_L}}(L')$, for non-isomorphic $L\not\cong L'$.

\end{remark}

Combining algebras of all ranks gives rise to a unified map
\[
\ch^\otimes: \mathbf{R_1}\to \Z[\seg]
\]
of same notation.

In light of Proposition \ref{prop:char-word}, we can alternatively present the reciprocal character in terms of an additive map 
\[
\widetilde{\mathrm{ch}}^\otimes: \mathbf{R_1}\to R
\]
that is given by
\[
\widetilde{\mathrm{ch}}^\otimes(V) = \sum_{L\in \Irr_1(n)} \frac{[V\,:\, \pi(L)]}{r_L} \; \pi(L)\;.
\]
The invariant $\widetilde{\mathrm{ch}}^\otimes(V)$ of a module $V$ is now visibly contains a partial information out of the full character $\mathrm{ch}(V)$.

\subsubsection{Integral decomposition of characters}\label{sect:nonintegral-qch}
The following discussion is included for completeness and can be skipped on a first reading.

Due to Proposition \ref{prop:bzring-decomp}, we may easily extend the reciprocal character map into an additive map
\[
\ch^\otimes: \mathbf{R} \to \Z\left[\widetilde{\seg}\right]\;,
\]
by setting $\alpha_\eta\circ \ch^\otimes\circ \sigma_{\eta^{-1}}^\ast: \mathbf{R}_\eta\to \Z[\seg(\eta)]$, on each tensor factor that is given by $\eta\in \C^\times$.

The observant reader may be wary at this point of the discrepancy with the definition of $\ch^\otimes$ supplied in the introduction section \eqref{eq:rec-ch}. Indeed, for the definitions to match, a notion of a standard representation $\zeta(\m)$, and of an indicator module $\zeta(\m)^{\otimes}$, should first be defined for any multisegment $\m\in \Z_{\geq0}^{\widetilde{\seg}}$.

Given such a monomial $\exp(\m) = \otimes_{i=1}^t\exp(\alpha_{\eta_i}(\m_i)) \in \Z\left[\widetilde{\seg}\right]$, with $0\neq \m_i\in \Z_{\geq0}^{\seg}$, for $\eta_1,\ldots,\eta_t\in \C^\times$, being the non-trivial factors of the decomposition in \eqref{eq:decomp-ring}, the induced module 
\[
\zeta(\m):= \sigma^\ast_{\eta_1}(\zeta(\m_1))\times\ldots \times \sigma^\ast_{\eta_t}(\zeta(\m_t))\in \Rep(H_n)
\]
is well-defined, and satisfies $\widetilde{\underline{\zeta}}(\m)= [\zeta(\m)]\in \mathbf{R}$, with respect to the Zelevinsky isomorphism of Section \ref{sect:zel}.

We may also write $\zeta(\m) = \iind_{\underline{n}}(\zeta^0(\m)^{\otimes})$, with the associated indicator module being defined as
\begin{equation}\label{eq:non-canon-ch}
\zeta^0(\m)^{\otimes} := \sigma_{\eta_1}^\ast(\zeta(\m_1)^\otimes)\boxtimes\cdots\boxtimes\sigma_{\eta_t}^\ast(\zeta(\m_t)^\otimes)\in \Rep(H_{\underline{n}})\;.
\end{equation}

It then follows, that for any $H_n$-module $V$, an equality of multiplicities
\begin{equation}\label{eq:prod-mult}
\left[  \rres_{\underline{n}}(V): \zeta^0(\m)^\otimes \right] = \prod_{i=1}^t [\rres_{\underline{n_{\m_i}}}(V_{\eta_i}):\zeta(\m_i)^{\otimes}]
\end{equation}
holds, where $[V] = \otimes_{\eta\in \C/\langle v^2\rangle} \sigma_{\eta}^\ast[V_{\eta}]\in \mathbf{R}$ is the decomposition of the module according to Proposition \ref{prop:bzring-decomp}.

Thus, our two definitions for the reciprocal character outside the integral case of $\mathbf{R}_1$ are visibly reconciled. Yet, it is worthwhile to note that the definition of an indicator module as in equation \eqref{eq:non-canon-ch} is to a large extent non-canonical, as explicated in the following proposition.

\begin{proposition}\label{prop:non-canon-ch}
Let $\m = \sum_{i=1}^t \alpha_{\eta_i}(\m_i)\in \Z_{\geq0}^{\widetilde{\seg}}$ be a multisegment given in the integral decomposition as above. Let us write 
\[
\zeta(\m_i)^\otimes = \zeta_{i,1}\boxtimes\cdots\boxtimes \zeta_{i,l_{\m_i}}\in \Rep_1(H_{\underline{n_{\m_i}}})
\]
for the integral indicator modules involved.

For any tableau $\rho = (\rho^i_j)\in \tab(\Z)$ that satisfies $I_{\rho} = \{(i,j): 1\leq i\leq t, 1\leq j\leq l_{\m_i}\}$ and $1\leq \rho^i_1<  \ldots < \rho^i_{l_{\m_i}}\leq K$, for all $1\leq i\leq t$, let us write the module
\[
\zeta^\rho(\m)^\otimes := 
\left(\bigtimes_{(i,j)\in I_{\rho}\::\; \rho^i_j=1} \sigma_{\eta_i}^\ast(\zeta_{i,j})\right)\boxtimes\cdots \boxtimes \left(\bigtimes_{(i,j)\in I_{\rho}\::\; \rho^i_j=K} \sigma_{\eta_i}^\ast(\zeta_{i,j})\right) \in \Rep(H_{\underline{n_{\rho}}})\;.
\]
Then, for any fixed choice of such a tableau $\rho =\rho(\m)$, we have $\zeta(\m) = \iind_{\underline{n_{\rho(\m)}}}(\zeta^{\rho(\m)}(\m)^{\otimes})$, and 
\[
\ch^\otimes(V) = \sum_{\m\in \Z_{\geq0}^{\widetilde{\seg}}}[\rres_{\underline{n_{\rho}}}(V)\,:\, \zeta^{\rho(\m)}(\m)^\otimes] \;\exp(\m) \in \mathbf{R}
\]
holds.

\end{proposition}

\begin{proof}
Since restriction functors are known to be compatible with the decomposition in Proposition \ref{prop:bzring-decomp}, the analogue of equation \eqref{eq:prod-mult} remains valid when $\zeta^0(\m)^{\otimes}$ is replaced with $\zeta^{\rho(\m)}(\m)^\otimes$.
\end{proof}

In light of this inherent non-canonicity, we opt to leave the notion of indicator modules flexible outside the integral case, that is, dependent on a choice of a tableau $\rho(\m)$ as in Proposition \ref{prop:non-canon-ch}. 

\subsection{Counting reciprocal multiplicities}

We would like to attain a combinatorial description for the reciprocal character of a given standard representation. Concretely, for given multisegments $\m,\n\in (\Z_{\geq0}^{\seg})_n$, we would like to describe the multiplicities
\[
A(\m,\n):= [\rres_{\underline{n_{\n}}}(\zeta(\m))\,:\, \zeta(\n)^\otimes]
\]
in terms of tableau counting.

\subsubsection{Adjusted indicator modules}\label{sect:adjust}
For a multisegment $\m\in (\Z_{\geq0}^{\seg})_n$, let us now assume that $\n \in  (\Z_{\geq0}^{\seg})_n$ is a multisegment that admits the inclusion $E(\n)\subseteq E(\m)$. In particular, $l_{\n}\leq l_{\m}$.

In this setup we write the simple \textit{adjusted} indicator module
\[
\zeta(\n)^{\otimes}_{(\m)}:= \zeta(\n[d_1])\boxtimes \cdots \boxtimes \zeta(\n[d_{l_{\m}}]) \in \Rep(H_{\underline{n_{\n,\m}}})\;,
\]
where $E(\m) = \{d_1<\ldots< d_{l_{\m}}\}$, and $\underline{n_{\n,\m}}\in C(n,l)$ (see Proposition \ref{prop:mackey}) equals the composition $\underline{n_{\n}}$ with possible addition of zero parts.

We still have $\zeta(\n) = \iind_{\underline{n_{\n,\m}}}(\zeta(\n)^{\otimes}_{(\m)})$, in similarity with the (non-adjusted) indicator module.

Furthermore, there is a clear identification of the algebra $H_{\underline{n_{\n,\m}}}$ with $H_{\underline{n_{\n}}}$, under which $\zeta(\n)^{\otimes}_{(\m)}$ corresponds to $\zeta(\n)^{\otimes}$. Yet, making this subtle distinction will assist us with the Mackey theory treatment.

For example, with $\m=[0,0]+[1,1]+[2,2]$ and $\n=[0,0]+[2,2]$, we have $\zeta(\n)^{\otimes}=Z([0,0])\boxtimes Z([1,1])\in \Irr(H_{(1,1)})$, while the adjusted indicator module is given as $\zeta(\n)^{\otimes}_{(\m)}=Z([0,0])\boxtimes \C \boxtimes Z([1,1])\in \Irr(H_{(1,0,1)})$.

\subsubsection{Mackey tableau counting}

\begin{lemma}\label{lem:Em-in-En}
    Whenever $A(\m,\n)$ is non-zero, we have $E(\n)\subseteq E(\m)$.
\end{lemma}

\begin{proof}

Let us write $\underline{\m} = [x_1,y_1]\cdots[x_l,y_l]\in \mathcal{W}(\seg)$ for a chosen ordering of $\m$, and $E(\n) = \{d_1<\ldots< d_{l_{\n}}\}$.

When $A(\m,\n)\neq0$, by Proposition \ref{prop:mackey} there must be a bitableau $(a^i_{j}, e^i_{j})_{i,j}=\mathcal{Q}\in \mathcal{K}(\underline{\m}, l_{\n})$ with $\zeta(\n)^{\otimes}$ appearing in $M(\mathcal{Q})\in \Rep(H_{\underline{n_{\n}}})$. In other words, the simple module $\zeta(\n[d_e])$ appears in $M(\m_{\mathcal{Q},e})$, for all $1\leq e\leq l_{\n}$. In particular, we see that $\supp(\m_{\mathcal{Q},e}) = \supp(\n[d_e])$.

Suppose now that $1\leq e\leq l_{\n}$ is fixed. The latter support condition implies that $\supp(\m_{\mathcal{Q},e})(d_e)>0$ and thus the existence of $(i,j)\in I_{\mathcal{Q}}^e$ with $d_e< a^i_{j-1}$ (when $a^i_0=y_i+1$ is taken). Since $a^i_{j-1}\leq a^i_0$ holds, we have $d_e\leq y_i$.

Since $y_i\in \supp(\m_{\mathcal{Q},e^i_1})$, we obtain that $y_i\leq d_{e^i_1}$. Noting that $e= e^i_j\geq e^i_1$, we now see that $d_{e^i_1}\leq d_e$, and $y_i\leq d_e$.

As a consequence, $d_e = y_i\in E(\m)$.

\end{proof}

For a given ordering $\underline{\m}=[x_1,y_1]\cdots[x_{l_{\m}},y_{l_{\m}}]\in \mathcal{W}(\seg)$ of a multisegment $\m\in \Z_{\geq0}^{\seg}$ with $E(\m)= \{d_1<\ldots <d_{l_{\m}}\}$, we define the sets of bitableaux
\[
\mathcal{K}(\underline{\m})_+ = \left\{\mathcal{Q}\in \mathcal{K}(\underline{\m},l_{\m})\::\:  \exists\n\in \Z_{\geq0}^{\seg},\; E(\n)\subseteq E(\m),\, \underline{n}_{\mathcal{Q}}= \underline{n_{\n,\m}} ,\,\left[M(\mathcal{Q})\;:\; \zeta(\n)_{(\m )}^\otimes\right] \neq 0\right\}\;,
\]
\[
\mathcal{K}(\underline{\m})_0 = \left\{ (a^i_j,e^i_j)_{i,j}=\mathcal{Q}\in \mathcal{K}(\underline{\m},l_{\m})\::\:  y_i\leq d_{e^i_1}, \forall1\leq i\leq l_{\m} \right\}\;.
\]

Through the same argument as in the proof of Lemma \ref{lem:Em-in-En}, it is easy to verify that $\mathcal{K}(\underline{\m})_+\subseteq \mathcal{K}(\underline{\m})_0$ holds.

\begin{lemma}\label{lem:Kunique}
For any bitableau $\mathcal{Q}\in \mathcal{K}(\underline{\m})_+$, there is a unique multisegment $\n\in\Z_{\geq0}^{\seg}$, for which the indicator module $\zeta(\n)_{(\langle\underline{\m}\rangle )}^\otimes$ is well-defined and appears in $M(\mathcal{Q})$.

We write $\n=\theta(\mathcal{Q})$, and thus obtain a map $\theta:\mathcal{K}(\underline{\m})_+\to \Z_{\geq0}^{\seg}$.

Moreover, $\left[M(\mathcal{Q})\;:\; \zeta(\theta(\mathcal{Q}))_{(\langle\underline{\m}\rangle )}^\otimes\right] =1$, for all $\mathcal{Q}\in \mathcal{K}(\underline{\m})_+ $.

\end{lemma}

\begin{proof}


Suppose that $\mathcal{Q}\in \mathcal{K}(\underline{\m})_+$ and $\n_1,\n_2\in\Z_{\geq0}^{\seg}$ are given with both $\zeta(\n_i)_{(\langle\underline{\m}\rangle )}^\otimes$, $i=1,2$, appearing in $M(\mathcal{Q})$. We have $\supp(\m_{\mathcal{Q},e}) = \supp(\n_i[d_e])$, for all $1\leq e\leq l$ and $i=1,2$.

Since right-aligned multisegments are spherical, we see that $\n_1[d_e] = \m_{\mathcal{Q},e}^\circ = \n_2[d_e]$, whenever those are non-zero. Uniqueness follows.

The multiplicity-one phenomenon follows from Lemma \ref{lemma:m_circ}\eqref{it:mc3}.

\end{proof}

\begin{corollary}\label{cor:Kcounting}
For any multisegments $\m,\n\in \Z_{\geq0}^{\seg}$ and an ordering $\underline{\m}$ of $\m$, we have
\[
A(\m,\n) = \#\left\{\mathcal{Q}\in \mathcal{K}(\underline{\m})_+ \;:\; \theta(\mathcal{Q}) = \n \right\} \;.
\]
\end{corollary}
\begin{proof}
Proposition \ref{prop:mackey} with Lemmas \ref{lem:Em-in-En} and \ref{lem:Kunique}.
\end{proof}

\subsubsection{Connection criterion}

We would like to state an explicit criterion for the identification of subset $\mathcal{K}(\underline{\m})_+$ within the full set $\mathcal{K}(\underline{\m}, l_{\langle \underline{\m}\rangle})$ of tableaux.

For a bitableau $(a^i_j,e^i_j)_{i,j}=\mathcal{Q}\in\mathcal{K}(\underline{\m},l_{\m})$, we make the standard convention of $a^i_0 = y_i+1$.

We write the set of indices 
\[
\widetilde{I}_{\mathcal{Q}} = \left\{(i,1)\in I_{\mathcal{Q}}\::\: y_i\neq d_{e^i_1}\right\}\;,
\]
and the shifted function $\widehat{\mathcal{Q}}: I_{\mathcal{Q}}\to \mathbb{Z}\times \mathbb{Z}$ that is given by $\widehat{\mathcal{Q}}(i,j) = (a^i_{j-1},e^i_j)$.

We say that an injective function $\tau: \widehat{I}_{\mathcal{Q}}\sqcup \widetilde{I}_{\mathcal{Q}}\to I_{\mathcal{Q}}$ is a \textit{Mackey connection} for $\mathcal{Q}$, when it satisfies $\mathcal{Q}\circ \tau = \widehat{\mathcal{Q}}$.

\begin{proposition}\label{prop:crit-mackey-conn}
Given an ordering $\underline{\m}$ of a multisegment $\m$ as above and a bitableau $\mathcal{Q}= (a^i_j, e^i_j)_{i,j}\in\mathcal{K}(\underline{\m},l_{\m})$, we have $\mathcal{Q}\in \mathcal{K}(\underline{\m})_+$, if and only if, a Mackey connection for $\mathcal{Q}$ exists.

When $\tau$ is a Mackey connection for $\mathcal{Q}$, we have 
   \[
    \theta(\mathcal{Q}) = \sum_{(i,j)\in I_{\mathcal{Q}} \setminus \tau \left(\widehat{I}_{\mathcal{Q}}\sqcup \widetilde{I}_{\mathcal{Q}}\right)} [a^i_j, d_{e^i_j}]\in \Z_{\geq0}^{\seg}\;.
    \]
\end{proposition}

\begin{proof}
Let us write the set of indices $S = \{1\leq e\leq l_{\m}\::\; \m_{\mathcal{Q},e}\neq0\}$, and for each $e\in S$, $m_e =\max\{i\;:\; \supp(\m_{\mathcal{Q},e})(i)\neq0\}$.

By Lemma \ref{lemma:m_circ}\eqref{it:mc3}, $Z\left(\m^{\circ}_{\mathcal{Q},1}\right)\boxtimes \cdots\boxtimes Z\left(\m^{\circ}_{\mathcal{Q},s}\right)$ appears in $M(\mathcal{Q})$. Thus, by uniqueness of spherical multisegments of a given support, $\mathcal{Q}\in\mathcal{K}(\underline{\m})_+$ holds, if and only if, all $\{\m_{\mathcal{Q},e}^{\circ}\}_{e\in S}$ are right-aligned multisegments and the sequence $\{m_e\}_{e\in S}$ is ascending.

By Proposition \ref{prop:crit-right}, for each $e\in S$, $\m_{\mathcal{Q},e}^{\circ}$ is right-aligned, if and only if, $b_e\in \Z$ exists, for which the invariant
\[
\beta(e,b):=\sum_{(i,j)\in I^e_{\mathcal{Q}}} \alpha_{a^i_j} - \sum_{(i,j)\in I^e_{\mathcal{Q}} \;:\; a^i_{j-1}-1\neq b } \alpha_{a^i_{j-1}} \in \Z^{\mathcal{I}}
\]
satisfies the positivity condition $\beta(e,b_e) \in \Z_{\geq0}^{\mathcal{I}}$. Moreover, when the condition holds, we have $\m_{\mathcal{Q},e}^{\circ}=\m_{\mathcal{Q},e}^{\circ}[b_e]$ and consequently $b_e=m_e$.

Suppose now that a Mackey connection $\tau: \widehat{I}_{\mathcal{Q}}\sqcup \widetilde{I}_{\mathcal{Q}}\to I_{\mathcal{Q}}$ exists.

Then, clearly for each $e\in S$, $\tau$ preserves $I^e_{\mathcal{Q}}$ and we have
\[
\beta(e, d_e) = \sum_{(i,j)\in I^e_{\mathcal{Q}}   } \alpha_{a^i_j} - \sum_{(i,j)\in I^e_{\mathcal{Q}}\cap(\widehat{I}_{\mathcal{Q}}\sqcup \widetilde{I}_{\mathcal{Q}} )} \alpha_{a^i_{j-1}}=\sum_{(i,j)\in I^e_{\mathcal{Q}}  } \alpha_{a^i_j} -  \sum_{(i,j)\in I^e_{\mathcal{Q}} \setminus \mathrm{Im}(\tau)  } \alpha_{a^i_j}  \in \Z_{\geq0}^{\mathcal{I}}\;.
\]
Hence, $\m_{\mathcal{Q},e}^{\circ}$ is right-aligned and $m_e= d_e$, for all $e\in S$. Since $\{d_e\}_{e\in S}$ is ascending, we obtain $\mathcal{Q}\in\mathcal{K}(\underline{\m})_+$.

Moreover, since $\theta(\mathcal{Q}) = \sum_{e=1}^{l_{\m}} \m_{\mathcal{Q},e}^{\circ}$, Proposition \ref{prop:crit-right} implies
\[
\theta(\mathcal{Q}) = \sum_{e\in S} \sum_{a\in \Z} \beta(e,d_e)(a)\cdot[a,d_e] =  \sum_{e\in S} \sum_{(i,j)\in I^e_{\mathcal{Q}} \setminus \mathrm{Im}(\tau)} [a^i_j,d_e]\;,
\]
which is equivalent to the stated formula.

Conversely, let us assume that $\zeta(\n)_{(\m)}^{\otimes}$ appears in $M(\mathcal{Q})$, for a given multisegment $\n\in \Z_{\geq0}^{\seg}$. 

Then, $\m_{\mathcal{Q},e}^\circ =\n[d_e]$ and $d_e = m_e$, for all $e\in S$. Once again, it follows from Proposition \ref{prop:crit-right} that $\beta(e,d_e)\in \Z_{\geq0}^{\mathcal{I}}$ holds, for all $e\in S$.

We may rewrite
\[
\beta(e,d_e) = \sum_{a\in \Z} \left( \# B(a,e) - \# \overleftarrow{B}(a,e) \right)\alpha_a\;,
\]
where
\[
    B(a,e) = \left\{(i,j)\in I^e_{\mathcal{Q}} : a^i_j=a \right\}, \quad 
    \overleftarrow{B}(a,e) =\left\{(i,j)\in I^e_{\mathcal{Q}}\cap(\widehat{I}_{\mathcal{Q}}\sqcup \widetilde{I}_{\mathcal{Q}} ) : a^i_{j-1}=a\right\}.
\]
Thus, the positivity condition amount to inequalities $\#\overleftarrow{B}(a,e)\leq \#B(a,e)$, for all $a\in \Z$ and $e\in S$. In particular, injective functions $\tau_{a,e}:\overleftarrow{B}(a,e) \hookrightarrow B(a,e)$ may be chosen, and then glued together into an injective function $\tau = \sqcup_{a,e}\tau_{a,e}$, which evidently serves as a Mackey connection.

\end{proof}

\section{Monomial bases in quantum groups}\label{sect:qgroups}

\subsection{Algebraic quantization}

We write $\mathbb{A} = \Z[q,q^{-1}]< \mathbb{Q}(q)$ for the formal ring of Laurent polynomials within the field of rational functions.

We write $\mathrm{ev}_1=\Z$ to be the $\mathbb{A}$-module given by function evaluation at $q=1$. Namely, a polynomial $p\in \mathbb{A}$ acts on $t\in \mathrm{ev}_1$ by $p.t=p(1)t$.

For a $\mathbb{A}$-module $V$, we write $V_{q\to 1}:= \mathrm{ev}_1\otimes_{\mathbb{A}}V$ for the additive group produced by tensoring modules. 

For a $\mathbb{A}$-module $V$, we also write $V^\ast:=\Hom_{\mathbb{A}}(V,\mathbb{A})$ for its dual module. 

Similarly, for a $\mathbb{A}$-morphism $\varphi:V\to W$, we write $\varphi_{q\to 1}:V_{q\to1}\to W_{q\to1}$ for the resulting additive map, and $\varphi^\ast:W^\ast\to V^\ast$ for the dual $\mathbb{A}$-morphism.

\subsubsection{Word $\mathbb{A}$-algebra}\label{sect:qgroups-word}
Writing $[k]:=\frac{q^k - q^{-k}}{q-q^{-1}}\in \mathbb{A}$, and $[k]!: = [k]\cdot[k-1]\cdot\ldots \cdot [1]\in \mathbb{A}$, for integers $k\geq1$, we define the \textit{divided powers} $\alpha^{(k)}:= \frac{1}{[k]!}\alpha^k\in R(q)$, for each letter $\alpha\in \mathcal{I}$.

Given a word $e\neq w\in \mathcal{W}(\mathcal{I})$ in its unique decomposition as $w = \alpha_{i_1}^{m_1}\cdot\ldots\cdot \alpha_{i_r}^{m_r}$, with $i_s\neq i_{s+1}$ and $m_s>0$, for each $1\leq s< r$, we define the element 
\[
w':= \alpha_{i_1}^{(m_1)}\cdot\ldots\cdot \alpha_{i_r}^{(m_r)}\in R(q)\;.
\]
We have $w' = \kappa_w^{-1}w$, for $\kappa_w = [m_1]!\cdots[m_r]!\in \mathbb{A}$. The convention $e'=e\in R_0(q)$ is set.

Clearly, $\mathscr{W}:= \{w'\}_{w\in \mathcal{W}(\mathcal{I})}$ is a $\mathbb{Q}(q)$-basis for $R(q)$, with $\mathscr{W}\{k\}:= \{w'\}_{w\in \mathcal{W}(\mathcal{I}_{\leq k})}$ spanning $R\{k\}(q)$, for each $k\geq0$.

We denote 
\[
R[q]= \bigoplus_{n=0}^\infty R_n[q] < R(q)
\]
to be the free $\mathbb{A}$-module generated by $\mathscr{W}$. Here, $R_n[q]= \mathrm{span}_{\mathbb{A}}\{w'\}_{w\in \mathcal{W}_n}$.

Analogously, we denote the submodules $R\{k\}[q] = \oplus_n R\{k\}[q]< R[q]$, with $R\{k\}[q] = \mathrm{span}_{\mathbb{A}}\{w'\}_{w\in \mathcal{W}^k_n}$.

\subsection{Quantum groups}\label{sect:qgroups}
Evidently, $R[q]$ is a sub-$\mathbb{A}$-algebra of $R(q)$, and it is a direct limit of the sequence of finitely generated $\mathbb{A}$-algebras $R\{k\}[q]$, as $k\to \infty$.

For $k\geq0$, Lusztig's $\mathbb{A}$-algebra $\mathbf{f}_{\mathbb{A}}\{k\}$ is defined (\cite[Section 1.2]{lusztig-quantum-book}, \cite[Section 4.1]{ramklesh}) as the quotient of $R\{k\}[q]$ by the ideal $S^k$ that is $\mathbb{A}$-generated by the \textit{quantum Serre relations}
\begin{equation}\label{eq:quantum-serre}
\begin{array}{ll}(\alpha_i \alpha_i\alpha_{j})'- (\alpha_i\alpha_{j}\alpha_i)' + (\alpha_j\alpha_{i}\alpha_i)' & ,\,|i-j|=1 \\
(\alpha_i\alpha_{j})'-(\alpha_j\alpha_i)' & ,\,|i-j|>1\;
\end{array}
,\quad|i|,|j|\leq k\;.
\end{equation}
Since those relations are homogeneous, the ideal $S^k= \bigoplus_{n=0}^\infty S^k_n$ is positively graded and the quotient $\mathbf{f}_{\mathbb{A}}\{k\}= \bigoplus_{n=0}^\infty \mathbf{f}_n\{k\}$ remains a graded algebra, with $\mathbf{f}_n\{k\}:= R_n\{k\}[q]/S^k_n$ being a quotient $\mathbb{A}$-module.

It is known that the complexified algebra $\mathbb{C}\otimes(\mathbf{f}_{\mathbb{A}}\{k\})_{q\to 1}$ is isomorphic with the positive part  $U(\mathfrak{sl}_{2k+2})^+$ of the corresponding type $A$ universal enveloping algebra.

\subsubsection{$A_\infty$ root system}
In better proximity to our setting, we take $\mathbf{f}_{\mathbb{A}}$ to be the quotient algebra of $R[q]$ by the $\mathbb{A}$-ideal $S= \bigoplus_{n=0}^\infty S_n$ that is generated by all quantum Serre relations, that is, \eqref{eq:quantum-serre} when taken without restriction by $k$.

We obtain that $\mathbf{f}_{\mathbb{A}} = \varinjlim_k \mathbf{f}_{\mathbb{A}}\{k\}$ is a direct limit algebra that produces $\mathbb{C}\otimes(\mathbf{f}_{\mathbb{A}})_{q\to 1}\cong  \varinjlim _{k} U(\mathfrak{sl}_{2k+2})^+$, a complex algebra often denoted as $U(\mathfrak{sl}_\infty)^+$.

Again, 
\[
\mathbf{f}_{\mathbb{A}}= \bigoplus_{n=0}^\infty \mathbf{f}_n\;,
\]
holds, where $\mathbf{f}_n = R[q]/S_n = \varinjlim_k \mathbf{f}_n\{k\}$.

For an element $f\in R_n\{k\}[q]$, we set $\overline{f}\in \mathbf{f}_n\{k\}<\mathbf{f}_n<\mathbf{f}_{\mathbb{A}}$ to denote its projection.

\subsubsection{Monomial basis}

There are several known constructions of bases for Lusztig's algebras, which easily extend to the $\mathfrak{sl}_\infty$ setting.

One such construction identifies suitable subsets of words in $\mathcal{W}(\mathcal{I})$, whose scalar multiples, as elements of $R[q]$, remain linearly independent when projected into $\mathbf{f}_{\mathbb{A}}$. Since such scalar multiples may be viewed as monomials in the free non-commutative $\mathbb{Q}(q)$-algebra $R(q)$, such bases are known as \textit{monomial bases}.

The following is derived out of one instance of monomial bases.

\begin{proposition}
The indicator words provide a basis 
\[
\mathscr{M}=\left\{m_{\m}:=\overline{w(\m)'}\in \mathbf{f}_{\mathbb{A}}\;:\;\m\in \Z_{\geq0}^{\seg}\right\}
\]
to $\mathbf{f}_{\mathbb{A}}$, as a free $\mathbb{A}$-module, that is labelled by the monoid of multisegments.
\end{proposition}

\begin{proof}
For $k\geq0$, \cite[Theorem 4.2 and remark after Lemma 4.5]{reineke-fei} gives 
\[
\mathscr{M}\cap \mathbf{f}_{\mathbb{A}}\{k\} = \left\{\overline{w(\m)'}\in \mathbf{f}_{\mathbb{A}}\;:\;\m\in \Z_{\geq0}^{\seg\{k\}}\right\}
\]
as a basis for $\mathbf{f}_{\mathbb{A}}\{k\}$. The claim now follows by taking a direct limit.
\end{proof}

We also denote the set $\mathscr{M}_n:= \mathscr{M}\cap \mathbf{f}_n = \{m_{\m}\}_{\m\in (\Z_{\geq0}^{\seg})_n}$, which becomes a $\mathbb{A}$-basis for $\mathbf{f}_n$.

\begin{remark}
The construction of subsets of words such as in the preceding proposition would typically depend on a choice of a decomposition of a longest Weyl element for the corresponding root system (or, of an orientation of the corresponding Dynkin diagram). Our discussion fixes one such canonical choice for type $A$ root systems.
\end{remark}

\subsection{Quantization of characters}

We would like to exhibit the $q$-variants of the character maps $\ch: G_1(H_n)\to R_n$ and $\ch^\otimes:G_1(H_n)\to \Z[\seg]_n$ from previous sections.

\subsubsection{Specialization maps}

While taking note of the basis $\{(w')^\ast\}_{w\in \mathcal{W}_n}$ for $R_n[q]^\ast$, that is dual to $\mathscr{W}\cap R_n[q]$, let us fix a natural group isomorphism
\[
\mathrm{sp}_{\mathscr{W}}: (R_n[q]^\ast)_{q\to 1}\overset{\sim}{\rightarrow} R_n\;,
\]
by setting $\mathrm{sp}_{\mathscr{W}}(1\otimes (w')^\ast) = \kappa_w(1)\cdot w$, for each word $w\in \mathcal{W}_n$. 

Put differently, for a morphism $\varphi:R[q]\to \mathbb{A}$ and a word $w\in \mathcal{W}(\mathcal{I})$, we have $\mathrm{sp}_{\mathscr{W}}(\varphi_{q\to1})(w) = \varphi(w)(1)$.

Clearly, $\mathrm{sp}_{\mathscr{W}}((R\{k\}[q]^\ast)_{q\to 1})= R\{k\}$ is valid, for each $k\geq0$.

The specification of the monomial basis $\mathscr{M}$ to $\mathbf{f}_{\mathbb{A}}$ provides a similar specialization map on the dual of Lusztig's algebra. Namely, we set a group isomorphism
\[
\mathrm{sp}_{\mathscr{M}}: (\mathbf{f}_{\mathbb{A}}^\ast)_{q\to 1}\overset{\sim}{\rightarrow} \Z[\seg]\;,
\]
by setting $\mathrm{sp}_{\mathscr{M}}(1\otimes m_{\m}^\ast) = \exp(\m)$, for each element of the dual monomial basis $\{m_{\m}^\ast\}_{\m\in \Z_{\geq0}^{\seg}}$.

\subsubsection{Canonical basis}

For each $\mathbf{f}_n$, $n\geq0$, Lusztig \cite{lusztig-canon} constructed a \textit{canonical} basis $\mathscr{B}_n$ as a $\mathbb{A}$-module.

The resulting basis $\mathscr{B}= \sqcup_n \mathscr{B}_n$ for $\mathbf{f}_{\mathbb{A}}$ is then specialized to a complex basis $\{1\otimes1\otimes  b \}_{b\in \mathscr{B}}$ in $U(\mathfrak{sl}_\infty)^+\cong \mathbb{C}\otimes_{\Z} \mathrm{ev}_1\otimes_{\mathbb{A}}\mathbf{f}_{\mathbb{A}}$ with favorable Lie-theoretic properties.

The link of the quantum group setting with our previous discussion of the representation theory of affine Hecke algebras appears through the following construction.

\begin{theorem}\label{thm:ch-quant}\cite[Section 6.8, Theorem 47]{leclerc-shuffle}(interpreting \cite{ariki-decomp})

For an integer $n\geq0$, let us consider the quotient $\mathbb{A}$-morphism $D_n: R_n[q]\to \mathbf{f}_n$.

There exists a unique bijection 
    \[
    \mathscr{B}_n \cong \Irr_1(n)\;,\quad b= b_L\in \mathscr{B}_n\;\leftrightarrow L = L_b\in \Irr_1(n)\;, 
    \]
between the canonical basis elements of the free $\mathbb{A}$-module $\mathbf{f}_n$ and the isomorphism classes of simple $H_n$-modules, such that the resulting group isomorphism that is given by
\[
\Phi_n:(\mathbf{f}^\ast_n)_{q\to 1} \overset{\sim}{\rightarrow} G_1(H_n) = \Z^{\Irr_1(n)}\;,\quad \Phi(1\otimes b^\ast) = L_b\;,
\]
where $\{b^\ast\}_{b\in \mathscr{B}_n}\subset \mathbf{f}^\ast_n$ is the dual basis to $\mathscr{B}_n$, makes the following diagram
\[
\begin{diagram}
\dgARROWLENGTH=2em
\node{ (\mathbf{f}^\ast_n)_{q\to 1} } \arrow{s,t}{\Phi_n }  \arrow{e,t}{(D_n^\ast)_{q\to1}}\node{  (R_n[q]^\ast)_{q\to1} }\arrow{s,t}{ \mathrm{sp}_{\mathscr{W}}}\\
\node{G_1(H_n) } \arrow{e,t}{ \Large \mathrm{ch} } \node{ R_n  }
\end{diagram}\;
\]
commute.

Equivalently, for a word $w\in \mathcal{W}_n$, when writing the expansion 
\[
\overline{w} = D_n(w) = \sum_{L\in \Irr_1(n)} c_{w,L} b_L\in \mathbf{f}_n\;,
\]
with $c_{w,L}\in \mathbb{A}$, we have
\[
c_{w,L}(1) = [L:w]\;,
\]
for all $L\in \Irr_1(n)$. 
\end{theorem}

The maps $\Phi_n$ in the last theorem can also be combined into a group isomorphism
\[
\Phi: (\mathbf{f}_\mathbb{A}^\ast)_{q\to1} \to \mathbf{R_1}\;.
\]

\begin{theorem}\label{thm:rec-ch}
The reciprocal character map $\ch^\otimes: \mathbf{R_1}\to \Z[\seg]$ coincides with the specialization at $q=1$ of the coordinate transition map between the duals of the canonical basis $\mathscr{B}$ and the monomoial basis $\mathscr{M}$ for $\mathbf{f}_{\mathbb{A}}$, when viewed through the Zelevinsky classification.

Namely, the commutation relation
\[
\begin{diagram}
\dgARROWLENGTH=2em
\node{ (\mathbf{f}^\ast_{\mathbb{A}})_{q\to 1} } \arrow{s,t}{\Phi }  \arrow{se,t}{ \mathrm{sp}_{\mathscr{M}}}\\
\node{\mathbf{R_1} } \arrow{e,t}{ \Large \ch^\otimes } \node{ \Z[\seg]  }
\end{diagram}\;
\]
holds, with the map $\Phi$ from Theorem \ref{thm:ch-quant}.

Equivalently, when expanding $m_{\m}= \sum_{L\in \Irr_1(n)} h_{\m,L} b_{L}\in \mathbf{f}_n$, with $h_{\m,L}\in \mathbb{A}$, for a multisegment $\m\in (\Z_{\geq0}^{\seg})_n$, we have
\[
h_{\m,L}(1)=[\rres_{\underline{n_{\m}}}(L)\,:\, \zeta(\m)^\otimes]\;,
\]
for all $L\in \Irr_1(n)$.
\end{theorem}

\begin{proof}
Since $\kappa_{w(\m)}\cdot m_{\m} =\overline{w(\m)}$, according to Theorem \ref{thm:ch-quant}, we have 
\[
\kappa_{w(\m)}(1)\cdot h_{\m,L}(1) = [L:w(\m)]\;.
\]
Thus, by Proposition \ref{prop:char-word}, it is enough to prove that $\kappa_{w(\m)}(1) = r_{\m}$.

Writing $w(\m)= w_{\beta^{\m}_1}\cdots w_{\beta^{\m}_l}$ as in Section \ref{sect:indicators}, it is evident that $\kappa_{w(\m)} = \kappa_{w_{\beta^{\m}_1}}\cdots \kappa_{w_{\beta^{\m}_l}}$. Hence, it is suffices to show that for any $\beta\in \Z_{\geq0}^{\mathcal{I}}$ with right-aligned $\m_{\beta}$,
\[
\kappa_{w_{\beta}}(1) = [Z(\m_\beta):w_\beta] = [M(\m_\beta):w_\beta]
\]
holds.

We prove it by induction on $n=|\beta|$.

Let us write $\underline{\m_{\beta}} = [a_1,b]\cdot\ldots\cdot[a_s,b]\in \mathcal{W}(\seg)$ for an ordering of $\m_{\beta}$, $\beta_0 = s\cdot \alpha_b = \epsilon(\m_{\beta})$ and $\gamma=\beta- \beta_0$. 

Then, $\m_\gamma = \sum_{i=1}^s[a_i,b-1]$ (with a convention of $[b,b-1]=0$) is right-aligned, $w_\beta = \alpha_b^s w_{\gamma} = w_{\beta_0} w_{\gamma}$, and $\kappa_{w_\beta} = [s]!\cdot \kappa_{w_\gamma}$.

Since $|\gamma| = n-s$, by induction it remains to show that $[M(\m_\beta):w_\beta] = s! [M(\m_\gamma):w_\gamma]$.

A straightforward application of Mackey theory (Proposition \ref{prop:mackey}) shows that $Z(\m_{\beta_0}) = M(\m_{\beta_0})= Z([b,b])^{\times s}$ has $\ch(Z(\m_{\beta_0})) = s!\cdot \alpha_b^s = s!\cdot w_{\beta_0}$.

It is also easy to see that there is a unique bitableau $\mathcal{Q} = \mathcal{Q}_1\cdots \mathcal{Q}_s\in \mathcal{K}\left(\underline{\m_{\beta}},2\right)$ with $\supp(\m_{\mathcal{Q},1}) = \supp(\m_{\beta_0})$. Namely, it is the tableau given by
\[
\mathcal{Q}_i = \left\{ \begin{array}{ll} (b,1)(a_i,2)  & a_i<b \\ (b,1) & a_i=b\end{array} \right.\;,
\]
for $i=1,\ldots,s$. It then follows that $\m_{\mathcal{Q},1}= \m_{\beta_0}$ and $\m_{\mathcal{Q},2}= \m_{\gamma}$.

Thus, when computing $\rres_{(s,n-s)}(M(\m_{\beta}))$ through Proposition \ref{prop:mackey}, we see that 
\[
[\rres_{(s,n-s)}(M(\m_{\beta})): M(\m_{\beta_0})\boxtimes M(\m_{\gamma})]=1\;.
\]

Finally, by Proposition \ref{prop:descending-spherical} and equation \eqref{eq:word-concat}, we see that
\[
\begin{array}{ll} [M(\m_\beta):w_\beta]& = [\rres_{(s,n-s)}(M(\m_{\beta})): M(\m_{\beta_0})\boxtimes M(\m_{\gamma})]\cdot [Z(\m_{\beta_0}):w_{\beta_0}]\cdot [Z(\m_\gamma):w_\gamma] =\\
 &= 1\cdot s!\cdot [Z(\m_\gamma):w_\gamma]\end{array}
\]
holds, as desired.
\end{proof}

\begin{corollary}
The reciprocal character map $\ch^\otimes: G_1(H_n)\to \Z[\seg]_n$ is a group isomorphism.

In particular, the semisimplification $[V]\in G_1(H_n)$ of any module $V\in \Rep_1(H_n)$ is determined by its reciprocal multiplicities 
\[
\left\{[\rres_{\underline{n_L}}(M):\zeta(L)^\otimes] \;:\; L\in \Irr_1(n)\right\}\;.
\]
\end{corollary}

\subsection{PBW bases}\label{sect:pbw}
We conclude this section by recalling another important class of bases of $\mathbf{f}_{\mathbb{A}}$, arising naturally in the theory of quantum groups. These are the quantum analogues of the classical \emph{PBW bases} of universal enveloping algebras.

A PBW basis consists of a compatible family of bases
\[
\mathscr{E}_n = \{e_{\m}\in \mathbf{f}_n\;:\;\m \in (\Z_{\geq0}^{\seg})_n\}\;,
\]
for each integer $n\geq0$, which satisfies the following proposition.

\begin{proposition}\label{prop:pbw}
\begin{enumerate}
    \item For each $k\geq0$, the subset $\mathscr{E}^k_n= \{e_{\m}\;:\;\m \in (\Z_{\geq0}^{\seg\{k\}})_n\}$ $\mathbb{A}$-generates the submodule $\mathbf{f}_n\{k\}<\mathbf{f}_n$. 
    \item For each $k\geq0$, the resulting basis $\mathscr{E}^k:= \bigsqcup_n \mathscr{E}^k_n$ for $\mathbf{f}\{k\}$ specializes to a PBW-basis $\left\{1\otimes 1\otimes e_{\m}\;:\; \m\in \Z_{\geq0}^{\seg\{k\}}\right\}$ for the universal enveloping algebra $U(\mathfrak{sl}_{2k+2})^+\cong \mathbb{C}\otimes_{\Z}\mathrm{ev}_1\otimes_{\mathbb{A}}\mathbf{f}\{k\}$.

\item\label{it:pbw3} There exists a labeling of the canonical basis $\mathscr{B}_n = \{b_{\m}\::\:\m\in (\Z_{\geq0}^{\seg})_n\}$ by multisegments, so that the expansions $b_{\n}= \sum_{\m\in (\Z_{\geq0}^{\seg})_n} a_{\n,\m} e_{\m}\in \mathbf{f}_n$ satisfy $a_{\m,\m}=1$, and $a_{\n,\m}\in v\Z[v]$, when $\n\neq\m$.

\end{enumerate}

\end{proposition}

The notion of a PBW basis is a standard part of the treatment of quantum groups in \cite{lusztig-canon,lusztig-quantum-book}. In fact, the typical approach to constructing the canonical basis $\mathscr{B}_n$ proceeds via a recursive construction of a PBW basis $\mathscr{E}_n$, after which $\mathscr{B}_n$ is uniquely characterized by the triangularity condition in Proposition~\ref{prop:pbw}\eqref{it:pbw3}.

We note that a PBW basis is not canonical: it depends on auxiliary Lie-theoretic choices, and thus varies within a natural family of possible bases. For our purposes though, we fix once and for all a PBW basis $\mathscr{E} = \bigsqcup_n \mathscr{E}_n$ of $\mathbf{f}_{\mathbb{A}}$, chosen so as to be compatible with our conventions in the representation theory of affine Hecke algebras and the Zelevinsky classification. The relevance and pinning of this choice are summarized in the following proposition.

\begin{proposition}\label{prop:kl}
\cite[Theorems 7,12]{lnt}

\begin{enumerate}

    \item For each integer $n\geq0$, the dual basis $\{e_{\m}^\ast\}\subset \mathbf{f}_n^\ast$ to $\mathscr{E}_n$ specializes to the collection of products of segment modules under the map $\Phi$ from Theorem \ref{thm:ch-quant}. More precisely, 
    \[
    \Phi(1\otimes e_{\m}^\ast) = M(\m) = [\zeta(\m)]\in G_1(H_n)
    \]
holds, for all $\m\in (\Z_{\geq0}^{\seg})_n$.

Equivalently, the diagram
\[
\begin{diagram}
\dgARROWLENGTH=2em
\node{ (\mathbf{f}^\ast_{\mathbb{A}})_{q\to 1} }
  \arrow{s,t}{\Phi}
  \arrow{se,t}{\mathrm{sp}_{\mathscr{E}}}
\\
\node{\mathbf{R_1}}
\node{ \Z[\seg] }
\arrow{w,b}{\underline{\zeta}}
\end{diagram}\;
\]
commutes, where $\mathrm{sp}_{\mathscr{E}}$ is the group isomorphism that is given $\mathrm{sp}_{\mathscr{E}}(1\otimes e_{\m}^\ast) = \exp(\m)$.

\item The labels of $\mathscr{B}$ that are provided by Theorem \ref{thm:ch-quant} and Proposition \ref{prop:pbw} agree through the Zelevinsky classification, i.e. for any multisegment $\m\in \Z_{\geq0}^{\seg}$, 
\[
b_{Z(\m)} = b_{\m}\in \mathscr{B}
\]
holds.
\item\label{it:kl3}
Consequently, the transition polynomials $a_{\n,\m}$ of Proposition \ref{prop:pbw}\eqref{it:pbw3} satisfy
\[
[M(\m):Z(\n)]=[\zeta(\m):Z(\n)] = a_{\n,\m}(1)\;,
\]
for all $\n,\m\in \Z_{\geq0}^{\seg}$.

\end{enumerate}

\end{proposition}

It is known that the transition polynomials $a_{\n,\m}$ of Proposition \ref{prop:pbw}\eqref{it:pbw3} admit a geometric interpretation in terms of intersection cohomology of quiver representation varieties of type $A$. With the above choice of PBW basis, they are naturally identified with parabolic Kazhdan–Lusztig polynomials associated with permutation groups.

\begin{corollary}
The polynomials $s_{\n,\m}\in \mathbb{A}$ that give the transition coefficients 
\[
m_{\n} = \sum_{\m\in (\Z_{\geq0}^{\seg})_n} s_{\n,\m} e_{\n}\in \mathbf{f}_n
\]
between the monomial basis $\mathscr{M}_n$ and the PBW basis $\mathscr{E}_n$, satisfy
\[
s_{\n,\m}(1) = A(\n,\m)\;.
\]

\end{corollary}
\begin{proof}
Writing the transition coefficients for the canonical basis as in Theorem \ref{thm:rec-ch} and Proposition \ref{prop:pbw}\eqref{it:pbw3}, we obtain
\[
s_{\m,\n} = \sum_{\mathfrak{u}\in (\Z_{\geq0}^{\seg})_n} h_{\m, Z(\mathfrak{u})} a_{\mathfrak{u},\n}\;.
\]
In light of Proposition \ref{prop:kl}\eqref{it:kl3}, evaluating the equation at $q=1$ gives
\[
s_{\m,\n}(1) = \sum_{\mathfrak{u}\in (\Z_{\geq0}^{\seg})_n}[\rres_{\underline{n_{\m}}}(Z(\mathfrak{u}))\,:\, \zeta(\m)^\otimes]\cdot [\zeta(\mathfrak{\n}):Z(\mathfrak{u})]\;.
\]

\end{proof}

\section{Quantum affine algebra modules and their $q$-characters}

\subsection{Limit $q$-character map}

\subsubsection{Leading-term dominant monoids}


Let $A=\bigsqcup_{i=1}^\infty A_i$ be a positively graded set.

For an element $0\neq \m\in \Z^A$, with $\m = \m_N + \m_{N-1} +\ldots + \m_1$, $\m_i\in \Z^{A_i}$, $i=1,\ldots, N$ and $\m_N\neq 0$, we define its degree to be $\deg(\m)=N$. 

We say that such $0\neq \m\in \Z^A$ is \textit{leading-term dominant}, when $\m_N\in \Z_{\geq0}^{A_N}$. The zero element $0\in \Z^{A}$ is assumed to be leading-term dominant of degree $0$.

The following proposition is easily verified.

\begin{proposition}\label{prop:dominated}

For a graded set $A$, the set of leading-term dominant elements 
\[
\Z_{\geq0}^A \subset \Z^{\{A\}}\subset \Z^A
\]
is a submonoid of $\Z^{A}$, and the equality $\deg ( \m + \n) = \max\{\deg(\m),\deg(\n)\}$ holds, for all $\m,\n\in \Z^{\{A\}}$.
\end{proposition}

\subsubsection{Leading-term dominant power series}

Let us assume that $A$ is a positively graded set and consider the intermediate monoid
\[
\mathscr{Y}_+(A) < \mathscr{Y}\{A\}:= \exp(\Z^{\{A\}})< \mathscr{Y}(A)\;. 
\]

We would like to define a larger ring than $\Z^{\mathscr{Y}\{A\}}$, where finiteness assumptions are relaxed.

The natural group embeddings $\mathscr{Y}(A_{\leq N})< \mathscr{Y}(A)$ give rise to projection of spaces 
\[
\mathrm{res}_N:\mathrm{Maps}(\mathscr{Y}\{A\},\Z)\to \mathrm{Maps}(\mathscr{Y}(A_{\leq N}),\Z)
\]
by restriction.

In these terms, we define the abelian group
\[
\Z[A]<\Z\{A\} := \{F\in \mathrm{Maps}(\mathscr{Y}\{A\},\Z)\;:\; \mathrm{res}_{N}(F)\in \Z[A_{\leq N}^{\pm 1}]\;, \forall N\geq1\}\;.
\]

For convenience, we will employ a notation in which elements of $\Z\{A\}$ are treated as power series, and are written in terms of infinite integer combinations of elements of $\mathscr{Y}(A)$.

For example, when treating $\seg = \sqcup_{i=1}^\infty \seg_i$ as a positively graded set, we may write $F= \sum_{b=1}^\infty \exp([0,b])\exp(-[1,b])\in \Z\{\seg\}$, so that $\mathrm{res}_N(F) = \sum_{b=1}^N \exp([0,b])\exp(-[1,b])$, for all $N\geq1$.

\begin{proposition}
For a positively graded set $A$, the additive group $\Z\{A\}$ has a ring structure that is given by
\[
(F_1\cdot F_2)(\exp(\m)) = \sum_{\n_1,\n_2\in \Z^{\{A\}}\,:\, \n_1+\n_2 = \m}F_1(\exp(\n_1))F_2(\exp(\n_2))\;,
\]
for $F_1,F_2\in \Z\{A\}$ and $\m \in \Z^{\{A\}}$, which naturally extends the polynomial ring structure on $\Z[A]$.

\end{proposition}

\begin{proof}
For any $\n_1,\n_2, \m= \n_1+ \n_2\in \Z^{\{A\}}$ with $\deg(\m)=N$, Proposition \ref{prop:dominated} implies that $\deg(\n_1),\deg(\n_2)\leq N$. 

By the assumed finiteness conditions on $\mathrm{res}_N(F_j)$, $j=1,2$, we then see that, when $\m$ is fixed, the set
\[
\{(\n_1,\n_2)\::\: \n_1+\n_2= \m\;,\; F_1(\exp(\n_1))F_2(\exp(\n_2))\neq 0\}
\]
is finite.

\end{proof}

\subsubsection{Limit $q$-character map}

For a segment $[x,y]\in \seg$, we consider the (infinite) set of words
\[
S(x,y) = \left\{
(a_{j}, b_{j})_{j=1}^{k}\in \mathcal{W}(\Z\times \Z) \,\middle|\,
\begin{array}{l}
x = a_{k} < \cdots < a_1 \leq y+1, \\
y =  b_1 < \cdots < b_{k} 
\end{array}
\right\},
\]

and the map 
\begin{equation}\label{eq:theta}
\theta: S(x,y) \to \Z^{\seg}\;,\quad \theta((a_{j}, b_{j})_{j=1}^{k}) = [a_1,y] + \sum_{j=2}^{k} \left([a_j,b_j]-[a_{j-1}, b_{j}]\right)\;,
\end{equation}
where $[y+1,y]=0$ is assumed, if necessary.

We now recall the positively graded structure $\seg = \sqcup_{N\geq1}\seg_N$ on the set of segments, that is given by the length parameter of each segment. We consider the ring $\Z\{\seg\}$ with respect to that grading.

Note, that for any word $w=(a_{j}, b_{j})_{j=1}^{k}\in S(x,y)$, we have $\theta(w)\in \Z^{\{\seg\}}$, since 
\[
\theta(w)- [a_k,b_k]\in \Z^{\seg_{\leq b_k-a_k-1}}
\]
holds, when $k>1$.

We define the \textit{limit $q$-character map} on segments as
\[
\widetilde{\qch}:\seg\to \Z\{\seg\}\,,\quad \widetilde{\qch}([x,y]) = \sum_{w\in S(x,y)} \exp(\theta(w))\;,
\]
and extend it to a ring homomorphism
\[
\qch: \Z[\seg]\to \Z\{\seg\}
\]
which is determined on ring generators by $\qch(\exp(\Delta)):= \widetilde{\qch}(\Delta)$, for all $\Delta\in \seg$.

\subsubsection{Bitableaux explication}\label{sect:explic}

For an ordered multisegment $\underline{\m} = [x_1,y_1]\cdots[x_r,y_r]\in \mathcal{W}(\seg)$, we consider the set of bitableaux
\[
\mathcal{J}(\underline{\m}) = \{\mathcal{P}_1\cdots \mathcal{P}_r\in \tab(\Z\times \Z) \,:\, \mathcal{P}_i\in S(x_i,y_i),\,\forall\,1\leq i\leq r\}\;.
\]

We can then write $\theta:\mathcal{J}(\underline{\m})\to \Z^{\{\seg\}}$ as
\begin{equation}\label{eq:theta2}
\mathcal{P}_1\cdots \mathcal{P}_r\in \mathcal{J}(\underline{\m})\;\mapsto\; \theta(\mathcal{P}_1)+\ldots +\theta(\mathcal{P}_r)\in \Z^{\{\seg\}}\;,
\end{equation}
and obtain the following description.

\begin{proposition}\label{prop:chq-description}
Let $\underline{\m}=[x_1,y_1]\cdots[x_r,y_r]$ be any ordering of a given multisegment $\m \in \Z_{\geq0}^{\seg}$.

Then,

\[
\qch(\exp(\m)) = \sum_{\mathcal{P}\in \JJ} \exp(\theta(\mathcal{P})) = 
\]
\[
=\sum_{ \mathcal{P}=(a^i_j,b^i_j)\in\JJ} \left( 
\prod_{i=1}^r\exp( [a^i_1,y_i])\prod_{(i,j)\in \widehat{I}_{\mathcal{P}} } \exp([a_j^i, b^i_j]) \exp(-[a^i_{j-1},b_j]) \right)
\]
holds in $\Z\{\seg\}$.

\end{proposition}

\subsubsection{Finite $q$-character maps}

 Given a subset $B\subset C$, we see a natural ring embedding
\[
\iota_{B,C}: \Z[B^{\pm1}]\to \Z[C^{\pm1}]\;,
\]
which preserves polynomials, that is, $\iota_{B,C}(\Z[B])< \Z[C]$. We also see the natural augmentation quotient map
\[
t_{C,B}: \Z[C^{\pm1}]\to \Z[B^{\pm1}]\,
\]
that is given as 
\[
t_{C,B}(\exp(\pm a)) = \left\{  \begin{array}{ll}  \exp(\pm a) &  a\in B \\ 1 & a\in C\setminus B \end{array} \right.
\]
on the generators $\{\exp(a),\exp(-a)\}_{a\in C}$ of $\Z[C^{\pm1}]$.

For each $N\geq1$, we consider the maps given by the formula 
\[
p_N:= t_{\seg_{\leq N+1}, \seg_N}\circ \mathrm{res}_{N+1}\;.
\]
We may treat it on two levels, as ring homomorphisms
\[
p_N: \Z[\seg]\to \Z[\seg_{\leq N}]\;,\quad p_N: \Z\{\seg\}\to \Z[\seg_{\leq N}^{\pm1}]\;,
\]
both of which are evidently surjective. On the polynomial ring level, we also take the embedding
\[
\iota_N:= \iota_{\seg_{\leq N},\seg}: \Z[\seg_{\leq N}]\hookrightarrow \Z[\seg]\;,
\]
which makes $p_N$ into a subspace projection operator (i.e. $p_N\circ \iota_N= \mathrm{id}$).

In these terms we define the \textit{$N$-finite $q$-character map} as the ring homomorphism
\[
\qchN:= p_N\circ\qch\circ\iota_N:\Z[\seg_{\leq N}]\to \Z[\seg_{\leq N}^{\pm1}]
\]
given by the composition
\[
\Z[\seg_{\leq N}]\xrightarrow{\iota_N} \Z[\seg]\xrightarrow{\qch} \Z\{\seg\}\xrightarrow{\mathrm{res}_{N+1}} \Z[\seg_{\leq N+1}^{\pm1}] \xrightarrow{t_{\seg_{\leq N+1}, \seg_N}} \Z[\seg_{\leq N}^{\pm1}]\;.
\]

\begin{lemma}\label{lem:commutepN}
    As homomorphisms $\Z[\seg]\to \Z[\seg_{\leq N}^{\pm1}]$, we have $\qchN\circ p_N = p_N\circ \qch$.
\end{lemma}
\begin{proof}
We need to verify the equality $p_N\circ\qch\circ\iota_N\circ p_N= p_N\circ \qch$. Since all maps involved are ring homomorphisms, it suffice to check the identity on generators $\exp([x,y])\in \Z[\seg] $.

When $[x,y]\in \seg_{\leq N}$, $\iota_N\circ p_N(\exp([x,y])) = \exp([x,y])$.

Suppose now that $y-x> N$, that is, $p_N(\exp([x,y]))=0$. We note that for each $w = (a_j,b_j)_{j=1}^k\in S(x,y)$, we have $b_k-a_k >N$ and thus, $p_N(\exp(\theta(w)))=0$. Hence, $p_N(\qch(\exp([x,y])))=0$ and $\exp([x,y])$ belongs to the kernels of both sides of the desired identity.

Finally, in the case of $y-x=N$, we have $p_N(\exp([x,y]))=1$. Through a similar argument as in the previous case, we see that a unique $w_0 = (x,y)\in S(x,y)$ with $p_N(\exp(\theta(w_0))\neq0$. Indeed, $\theta(w_0) = [x,y]$ and $p_N(\qch(\exp([x,y]))) = p_N (\exp([x,y]))=1$.

\end{proof}

\subsection{Representation theory of quantum affine algebras}

For a transcendental parameter $v\in \C^\times$ and integer $N\geq 1$, we consider the Drinfeld--Jimbo quantum affine algebra 
\[
U_{N+1} = U_{v}\!\left(\widehat{\mathfrak{sl}_{N+1}}\right)\;.
\]
This is the complex associative algebra that is generated by a Laurent polynomial algebra $A_{N+1}=\mathbb{C}[k_0^{\pm 1},k_1^{\pm},\ldots, k_{N}^{\pm 1}]$ together with additional generators $\{x^+_i, x^-_i\}_{i=0}^N$, subject to the relations

\[
\begin{array}{ll} x^{\pm}_i x^{\pm}_j = x^{\pm}_j x^{\pm}_i , & |i-j|\neq 1\;(\mathrm{mod}\; N) \\ 
x^{+}_i x^{-}_j = x^{-}_j x^{+}_i , & i\neq j \\
\left(x^{\pm}_i\right)^2 x^{\pm}_j - (v+v^{-1})x^{\pm}_ix^{\pm}_jx^{\pm}_i + x^{\pm}_j\left( x^{\pm}_i\right)^2=0\;, & |i-j|=1 \;(\mathrm{mod}\; N) \\
x^{+}_i x^{-}_i - x^{-}_i x^{+}_i = \frac{1}{v - v^{-1}}(k_i-k_i^{-1}) \;,\\
k_jx^{\pm}_i = x^{\pm}_ik_j , & |i-j|\neq 0,1\;(\mathrm{mod}\; N) \\ k_ix^{\pm}_i k_i^{-1} = v^{\pm 2}x^{\pm}_i , &  \\ k_jx^{\pm}_i k_j^{-1} = v^{\mp 1}x^{\pm}_i& |i-j|=1 \;(\mathrm{mod}\; N)\end{array} \;.
\]



A finite-dimensional module over $U_{N+1}$ is said to be of \textit{type $1$}, when the sub-algebra $A_{N+1}$ acts on it semisimply with all $k_i$-eigenvalues being integer powers of $v$, for $i=0,\ldots N$.

The standard framework for the representation theory of quantum affine algebras is the category $\mathcal{C}_N$ of finite-dimensional type $1$ $U_{N+1}$-modules. It is known (\cite[Proposition 2.6]{Chari-Pressley}) that every such module factors through the well-studied level $0$ quotient of $U_{N+1}$, commonly referred to as the \textit{quantum loop algebra}. Moreover, up to a finite group of algebra automorphisms, the class of type $1$ modules accounts for all finite-dimensional modules (see \cite[Proposition 12.2.3]{cp-book}).

The quantum affine algebra $U_{N+1}$ carries a natural Hopf algebra structure, which endows $\mathcal{C}_N$ with the structure of a monoidal category, with this tensor product denoted by us as $\otimes$.

In particular, the $1$-dimensional module that we denote as $\mathbf{1}\in \mathcal{C}_N$, on which $x_i^{\pm}$'s act by $0$ and $k_i$'s act by $1$, is the unit object in the category.

Let $\Irr[N]$ denote the set of isomorphism classes of simple modules in $\mathcal{C}_N$, and set
\[
G[N]=\Z^{\Irr[N]}
\]
for the Grothendieck group of $\mathcal{C}_N$. 

\subsubsection{Grothendieck rings}

The monoidal structure induces on $G[N]$ the structure of an associative ring.


We consider the formal set
\[
\mathscr{D}_N = \left\{\widetilde{Y}(i,a)\,:\, i\in \{1,\ldots, N\}, a\in \mathbb{C}^\times\right\}\;,
\]
of parameters, and write $Y(i,a):=\exp\left(\widetilde{Y}(i,a)\right)\in \Z[\mathscr{D}_N]$ for the generators of the resulting polynomial ring.

Each $\widetilde{Y}\in \mathscr{D}_N$ gives rise to a \textit{fundamental module} $L\left(\widetilde{Y}\right) = L(\exp(\widetilde{Y}))\in \Irr[N]$, as defined in \cite[Section 1.3]{fm2001} or \cite[Theorem 5.17]{kashiwara-duke}\footnote{More precisely, fundamental modules were defined as $L(Y(i,1)) $ in our notation, while $L(Y(i,a)):= \tau_{a}^\ast(L(Y(i,1)))$ maybe taken as the general definition, with $\tau_a^\ast$ being the twist as used in Section \ref{sect:qa-integral}.}, for example.

When writing $S(\widetilde{Y}) = [L(\exp(\widetilde{Y}))]\in G[N]$, the map $S: \mathscr{D}_N\to G[N]$ can now be extended into a ring homomorphism
\[
S: \Z[\mathscr{D}_N ]\to G[N]\;,
\]
that is determined by $S(Y(i,a))= [L(Y(i,a))]$.

\begin{proposition}\label{prop:poly-ring}
\cite{MR1745260}\cite[Theorem 2.2(4)]{fm2001}

The map $S$ is a ring isomorphism.

In particular, the ring $G[N]$ is commutative, and it is a polynomial ring in the set of variables given by the classes of fundamental modules.
\end{proposition}

\subsubsection{Integral decomposition}\label{sect:qa-integral}
For $\eta\in \C^\times$, let us consider the parameter subset
\[
\mathscr{D}_{N,\eta} = 
\left\{\widetilde{Y}(i,a)\in \mathscr{D}_N\::\: a = \eta v^p,\; p+i+1\in 2\Z\right\}\;.
\]
We write $\alpha_\eta:\mathscr{D}_N\to \mathscr{D}_N$ for bijective function $\alpha_\eta(\widetilde{Y}(i,a)) = \widetilde{Y}(i,\eta a)$, so that $\alpha_\eta(\mathscr{D}_{N,1})= \mathscr{D}_{N,\eta}$. Those map can also be extended naturally to ring automorphisms of $\Z[\mathscr{D}_N]$.

We denote the subring $G_{\eta}[N]:= S(\Z[\mathscr{D}_{N,\eta} ])<G[N]$, which may be viewed as the Grothendieck group of a Serre subcategory $\mathcal{C}_{N,\eta}$ in $\mathcal{C}_N$. We also denote $\Irr_{\eta}[N] \subset \Irr[N]$ to be the corresponding subset of simple modules in $\mathcal{C}_{N,\eta}$, that is, $G_\eta[N] = \Z^{\Irr_\eta[N]}$.

In similarity with \eqref{eq:decomp-ring}, we have a natural decomposition
\begin{equation}\label{eq:qa-tensor}
\Z\left[\mathscr{D}_N\right] = \bigotimes_{\eta\in \C^\times/\langle v^2\rangle} \Z[\mathscr{D}_{N,\eta}]
\end{equation}
of polynomial rings, as well as an analogous decomposition of rings of Laurent polynomials.

It is known (\cite[Theorem 1.3(4)]{fm2001}) that a family of algebra automorphisms $\{\tau_{\eta}\}_{\eta\in \C^\times}$ exists for $U_{N+1}$, pulling modules through which satisfies
\begin{equation}\label{eq:intertwUeta}
\tau_{\eta}^\ast\circ S = S\circ \alpha_\eta\;.
\end{equation}
Thus, $\tau_\eta^\ast: \mathcal{C}_{N,1}\xrightarrow{\sim} \mathcal{C}_{N,\eta}$ serves as a monoidal equivalence of categories.

\subsubsection{q-characters}
An invariant of high interest in the study of this theory is the $q$-character. For a module $V\in \mathcal{C}_N$, its $q$-character is presented as a Laurent polynomial
\[
\chi^N(V)\in \Z[\mathscr{D}_N^{\pm1}]\;,
\]
which is well-defined \cite[Section 2]{fm2001} on the semisimplification $[V]\in G[N]$, i.e. $\chi^N(V) = \chi^N([V])$. As such, it provides a ring homomorphism
\[
\chi^N: G[N]\to \Z[\mathscr{D}_N^{\pm1}]\;.
\]

It is known (see, for example, explicit formulas below) that for a fundamental module $L=L\left(\widetilde{Y}\right)\in \mathcal{C}_N$ with $\widetilde{Y}\in \mathscr{D}_{N,\eta}$, we have $\chi^N(L)\in \Z[\mathscr{D}_{N,\eta}^{\pm1}]$. Since those are the ring generator, we see that the $q$-character map restricts to homomorphisms
\[
\chi^N: G_{\eta}[N]\to \Z[\mathscr{D}_{N,\eta}^{\pm1} ]\;,
\]
for all $\eta\in \C^\times$.

Moreover, it is known (\cite[Proposition 4.7]{henandez-simple}) that
\begin{equation}\label{eq:qch-twist}
\alpha_{\eta}\circ \chi^N = \chi^N\circ \tau_\eta^\ast
\end{equation}
holds, for all $\eta\in \C^\times$. Thus, $\chi^N$ is determined by its values on $G_1[N]$.

\subsubsection{Identification with truncated $q$-characters}\label{sect:phimap}

There is a bijection between the sets $\seg_{\leq N}$ and $\mathscr{D}_{N,1}$ that is given by
    \[
                [x,y] =\left[\frac{p-l+1}{2},\frac{p+l-1}{2}\right]\in \seg_{\leq N}  \quad \longleftrightarrow\quad \widetilde{Y}(y-x+1,v^{x+y})= \widetilde{Y}(l,v^p)\in \mathscr{D}_{N,1} \;.
        \]

The bijection leads to an identification of polynomial rings $\phi:\Z[\mathscr{D}_{N,v}^{\pm1}]\xrightarrow{\sim} \Z[\seg_{\leq N}^{\pm1}]$, that is given by $\phi(Y(l,v^p))  = \exp([x,y])$ on generators.

The following theorem is an interpretation of the explicit formula for $q$-character that was given by Nakajima in \cite{naka-formula}.

\begin{theorem}\label{thm:qchar-ident}
    Under the above identification of parameter sets, we have an equality $\chi^N \circ S = \phi^{-1}\circ\qchN \circ\phi $ of ring homomorphisms, i.e. the following diagram commutes.
\[
\begin{tikzcd}[column sep=large,row sep=0.6cm]
\Z[\seg_{\leq N}] \arrow[r, "\qchN"] \arrow[d, "\phi^{-1}"] & \Z[\seg_{\leq N}^{\pm1}] \arrow[dd, "\phi^{-1}"] \\
\Z[\mathscr{D}_{N,1}]       \arrow[d, "S"]                     &                   \\
G_1[N] \arrow[r, "\chi^N"]                   & \Z[\mathscr{D}_{N,1}^{\pm1}]           
\end{tikzcd}
\]
\end{theorem}

\begin{proof}
It is enough to verify the equality on the ring generators $Y(l,v^p)\in \Z[\mathscr{D}_{N,1}]$. In other words, we need to show that $\phi(\chi^N(L(Y(l,v^p))))=\qchN(\exp([x,y]))$, when $(x,y) = (\frac12(p-l+1),\frac12(p+l+1))$.

Let us first recall the explicit formula for $q$-characters of fundamental modules that was given in \cite[Proposition 4.6]{naka-formula}:
\[
      \chi^N(L(Y(l,v^p)))=
      \]
      \[
 =     \sum_{1\leqslant c_1<\ldots<c_l\leqslant N+1}  \prod_{m=1}^l Y(c_m,v^{p+l+c_m-2m})Y(c_m-1, v^{p+l+c_m-2m+1})^{-1}\in \Z[\mathscr{D}_{N,1}^{\pm1} ]\;,
\]
where $Y(0,a)= Y(N+1,a)=1$ is assumed.

The last formula may be translated into
\[
\phi(\chi^N(L(Y(l,v^p)))) = t_{\seg_{\leq N+1}, \seg_N}\left( \sum_{y\leq b_1\leq \ldots\leq b_l \leq x+ N} \prod_{m=1}^{l} \exp\left(   [y-m+1,b_j]-[y-m+2, b_{j}] \right)\right)\;,
\]
in $\Z[\seg^{\pm 1}_{\leq N}]$, with $[y+1,y] = 0$ being assumed.

When accounting for cancellations that occur in the last formula, in cases of $b_m = b_{m+1}$, we obtain
\[
\phi(\chi^N(L(Y(l,v^p)))) =  t_{\seg_{\leq N+1}, \seg_N}\left(\sum_{w\in S(x,y)\::\: \deg(\theta(w))\leq N+1} \exp(\theta(w)) \right) =
\]
\[
= t_{\seg_{\leq N+1}, \seg_N} \circ \mathrm{res}_{N+1}\circ \qch( \exp[x,y])\;.
\]

\end{proof}

\subsection{Dominant monomials}

We would like to explore the \textit{dominant} part of the limit $q$-character, that is, the map given by
\[
\qchp:\Z[\seg] \to \Z\{\Seg\}\,,\quad \qchp(F): = \qch(F)_+\;,
\]
which is not a ring homomorphism.

\begin{lemma}\label{lem:coheren-plus}
    The map $\qchNp:\Z[\seg_{\leq N}]\to \Z[\seg_{\leq N}^{\pm1}]$ that is defined by the composition
\[
\Z[\seg_{\leq N}]\xrightarrow{\iota_N} \Z[\seg]\xrightarrow{\qchp} \Z\{\seg\}\xrightarrow{\mathrm{res}_{N+1}} \Z[\seg_{\leq N+1}^{\pm1}] \xrightarrow{t_{\seg_{\leq N+1}, \seg_N}} \Z[\seg_{\leq N}^{\pm1}]
\]
satisfies $\qchNp(F) =\qchN(F)_+$, for all $F\in\Z[\seg_{\leq N}]$.

In particular, its image lies in the polynomial ring $\Z[\seg_{\leq N}]$.

As additive maps $\Z[\seg]\to \Z[\seg_{\leq N}]$, we have $\qchNp\circ p_N= p_N \circ \qchp$.
\end{lemma}

\begin{proof}
For any $F\in\Z[\seg]$, we have $\mathrm{res}_{N+1}(\qch(F))\in \Z^{\mathscr{Y}\{\seg\}\cap \mathscr{Y}(\seg_{\leq N+1})}$. Since only leading-term dominant monomials are involved, we have 
\[
t_{\seg_{\leq N+1}, \seg_N}( \mathrm{res}_{N+1}(\qch(F)))_+=t_{\seg_{\leq N+1}, \seg_N}( \mathrm{res}_{N+1}(\qch(F))_+)\;.
\]
The first statement now follows, since the restriction map $\res_{N+1}$ clearly commutes with the operation of taking the dominant part.

For the last statement, we see that
\[
\qchNp(p_N(F)) = \qchN(p_N(F))_+ = p_N(\qch(F))_+ = p_N(\qchp(F))
\]
holds, where the middle equality is given by Lemma \ref{lem:commutepN}.
\end{proof}

Indeed, for a module $V\in \mathcal{C}_N$, the monomials of $\mathscr{Y}_+(\mathscr{D}_N)$ that appear in the $q$-character $\chi^N(V)$ are known as the \textit{dominant weights} of $V$ and are a source of interest. 

In light of Lemma \ref{lem:coheren-plus}, Proposition \ref{prop:poly-ring} and Theorem \ref{thm:qchar-ident}, when viewing the module $V$ as an element $F:=\phi(S^{-1}([V]))\in \Z[\seg_{\leq N}]$, the count of the dominant weights spaces for $V$ is reduced to study of the polynomial $\qchNp(F)$, which in turn, is a truncated version of $\qchp(F)$.

\subsubsection{Dominant bitableaux}

For an ordered multisegment $\underline{\m}\in \mathcal{W}(\seg)$, we consider the subset of bitableaux
\[
\JJ_+:= \left\{\mathcal{P}\in \JJ\::\: \theta(\mathcal{P})\in \Z_{\geq0}^\seg\right\}\;.
\]
Thus, for any ordering $\underline{\m}$ of a multisegment $\m\in \Z_{\geq0}^{\seg}$, we have
\[
\qchp(\exp(\m))=  \sum_{\mathcal{P}\in \JJ_+} \exp(\theta(\mathcal{P}))\;.
\]

We would like to develop a combinatorial condition for detection of the subset $\JJ_+$ within the full set $\JJ$ of bitableaux.

For a bitableau $\mathcal{P}= (a^i_j, b^i_j)_{i,j}\in \JJ$, we write the shifted function $\widehat{\mathcal{P}}: \widehat{I}_{\mathcal{P}}\to \mathbb{Z}\times \mathbb{Z}$ that is given by $\widehat{\mathcal{P}}(i,j) = (a^i_{j-1},b^i_j)$.

We say that an injective function $\tau:\widehat{I}_{\mathcal{P}}\to I_{\mathcal{P}}$ is a \textit{dominant connection} for $\mathcal{P}$, when it satisfies $\mathcal{P}\circ \tau = \widehat{\mathcal{P}}$.

\begin{theorem}\label{thm:connection}
For an ordered multisegment $\underline{\m}\in \mathcal{W}(\seg)$ and a bitableau $\mathcal{P}= (a^i_j, b^i_j)_{i,j}\in \JJ$, we have $\mathcal{P}\in \JJ_+$, if and only if, a dominant connection for $\mathcal{P}$ exists.

When such a dominant connection $\tau$ exists, we have
    \[
    \theta(\mathcal{P}) = \sum_{(i,j)\in I_{\mathcal{P}} \setminus \tau \left(\widehat{I}_{\mathcal{P}}\right)} [a^i_j, b^i_j]\in \Z_{\geq0}^{\seg}\;.
    \]
\end{theorem}

\begin{proof}
Let $\mathcal{P}\in \JJ$ be a given bitableau.

For $\mathcal{P}(i,j)=(x,y)$, with $(i,j)\in I_{\mathcal{P}}$, let us write $[\mathcal{P}(i,j)]:=[x,y]\in \seg$, when $x\leq y$, and $[\mathcal{P}(i,j)]=0\in \seg$, when $x= y+1$.

Equations \eqref{eq:theta} and \eqref{eq:theta2} can be written as
\begin{equation}\label{eq:theta3}
\theta(\mathcal{P}) = \sum_{(i,j)\in I_{\mathcal{P}}} [\mathcal{P}(i,j)]- \sum_{(i,j)\in \widehat{I}_{\mathcal{P}}} [\widehat{\mathcal{P}}(i,j)]\in \Z^{\seg}\;.
\end{equation}

Suppose now that $\tau$ is a $\mathcal{P}$-connection. Then,
\[
\theta(\mathcal{P}) = \sum_{(i,j)\in I_{\mathcal{P}}} [\mathcal{P}(i,j)]- \sum_{(i,j)\in \widehat{I}_{\mathcal{P}}} [\mathcal{P}(\tau(i,j))]= \sum_{(i,j)\in I_{\mathcal{P}} \setminus \tau \left(\widehat{I}_{\mathcal{P}}\right)} [\mathcal{P}(i,j)]\in \Z_{\geq0}^{\seg}\;,
\]
and one direction of the claim is settled.

Conversely, suppose that $\mathcal{P}\in \JJ_+$, that is, $\theta(\mathcal{P})\in \Z_{\geq0}^{\seg}$, holds.

We may rewrite \eqref{eq:theta3} once more, in the form of
\[
\theta(\mathcal{P}) = \sum_{\Delta \in \seg} \left( \# A_{\Delta}(\mathcal{P}) - \# \overleftarrow{A}_{\Delta}(\mathcal{P}) \right)\Delta\;,
\]
where
\[
    A_{\Delta}(\mathcal{P}) = \left\{(i,j)\in I_{\mathcal{P}} : [\mathcal{P}(i,j)]= \Delta \right\}, \quad 
    \overleftarrow{A}_{\Delta}(\mathcal{P}) =\left\{(i,j)\in \widehat{I}_{\mathcal{P}} : [\widehat{\mathcal{P}}(i,j)]= \Delta \right\}.
\]
Thus,
\[
\#A_{\Delta}(\mathcal{P})-\#\overleftarrow{A}_{\Delta}(\mathcal{P}) \ge 0
\]
holds by assumption, for any segment $\Delta\in \seg$. In particular, we may choose injections $\tau_{\Delta}:\overleftarrow{A}_{\Delta}(\mathcal{P})\hookrightarrow A_{\Delta}(\mathcal{P})$, and assemble them into a single injective map
\[
\tau:\widehat{I}_{\mathcal{P}}\hookrightarrow 
I_{\mathcal{P}}, \qquad
\tau|_{\overleftarrow{A}_{\Delta}(\mathcal{P})}=\tau_{\Delta},\;\forall\Delta\in \seg\;,
\]
when bearing in mind that the index set $\widehat{I}_{\mathcal{P}}$ is a disjoint union over $\left\{\overleftarrow{A}_{\Delta}(\mathcal{P})\right\}_{\Delta\in \seg}$.

By construction, we now have $\mathcal{P}(\tau(i,j))=\widehat{\mathcal{P}}(i,j)$, for all $(i,j)\in \widehat{I}_{\mathcal{P}} $.

\end{proof}

For an ordered multisegment $\underline{\m}\in \mathcal{W}(\seg)$, let us consider the subsets
\[
\JJ_0 = \left\{\mathcal{P}= (a^i_j, b^i_j)_{i,j}\in \JJ\::\: \,b^i_j\in E(\langle \m\rangle )\;, \forall (i,j)\in I_{\mathcal{P}} \right\}\;,
\]
\[
\JJ_{00} = \left\{\mathcal{P}= (a^i_j, b^i_j)_{i,j}\in \JJ\::\: a^i_j\in D(\langle \m\rangle ),\,b^i_j\in E(\langle \m\rangle )\;, \forall (i,j)\in I_{\mathcal{P}} \right\}\subseteq \JJ_0
\]
of bitableaux, which are evidently finite.

\begin{proposition}\label{prop:+0}
For an ordered multisegment $\underline{\m}\in \mathcal{W}(\seg)$, an inclusion $\JJ_+\subseteq \JJ_{00}$ holds.
\end{proposition}
\begin{proof}
Let $\mathcal{P}= (a^i_j, b^i_j)_{i,j}\in \JJ_+$ be a given bitableau, and $\tau$ a $\mathcal{P}$-connection, which exists by Theorem \ref{thm:connection}.

Let $(i_0,j_0)\in \widehat{I}_{\mathcal{P}}$ be a fixed index. We denote $\tau^k(i_0,j_0)=(i_k,j_k)\in I_{\mathcal{P}}$, for $1\leq k$ where the index is defined. 

For such $1\leq k$, we have 
\[
(a^{i_{k-1}}_{j_{k-1}-1},b^{i_{k-1}}_{j_{k-1}})=\widehat{\mathcal{P}}(i_{k-1},j_{k-1}) = \mathcal{P}(i_k,j_k) = (a^{i_k}_{j_k},b^{i_k}_{j_k})\;.
\]
Thus, $a^{i_{k-1}}_{j_{k-1}}<a^{i_{k-1}}_{j_{k-1}-1}=a^{i_k}_{j_k}$ and $b^{i_{k-1}}_{j_{k-1}}=b^{i_k}_{j_k}$. In particular, we see that $a^{i_0}_{j_0}< a^{i_k}_{j_k}$, and that $\tau$ cannot form of a cycle in $\widehat{I}_{\mathcal{P}}$, in the sense that there cannot exists $1\leq k$, for which $(i_k,j_k) = (i_0,j_0)$.

Therefore by finiteness of $\widehat{I}_{\mathcal{P}}$, $1\leq k_{\max}$ must exist, for which $(i_{k_{\max}},j_{k_{\max}}) \in I_{\mathcal{P}}\setminus \widehat{I}_{\mathcal{P}}$. 

Hence, $j_{k_{\max}}=1$ and $b^{i_0}_{j_0} = b^{i_{k_{\max}}}_1 \in E(\langle \underline{\m}\rangle )$ by definition of $\JJ$.

Let us now consider the index set
\[
\widetilde{I}_{\mathcal{P}} = \{(i,j)\in I_{\mathcal{P}}\::\: j<j(i)\}\;,
\]
where $j=j(i)$ is the maximal number with $(i,j)\in I_{\mathcal{P}}$, and the injective map $\tau_+:\widetilde{I}_{\mathcal{P}} \to I_P$, given by $\tau_+(i,j) :=\tau(i,j+1)$.

We denote $(\tau_+)^k(i_0,j_0)=(i_{-k},j_{-k})\in I_{\mathcal{P}}$, for $1\leq k$ where the index is defined. 

Now, since
\[
(a^{i_{-k+1}}_{j_{-k+1}},b^{i_{-k+1}}_{j_{-k+1}+1})=\widehat{\mathcal{P}}(i_{-k+1},j_{-k+1}+1) = \mathcal{P}(i_{-k},j_{-k}) = (a^{i_k}_{j_k},b^{i_{-k}}_{j_{-k}})
\]
holds for such $1\leq k$, by similar arguments, we must have $k_{\min}\leq 0$ with $j_{k_{\min}}= j(i_{k_{\min}})$ and $a^{i_0}_{j_0}= a^{i_{k_{\min}}}_{j_{k_{\min}}}\in D(\langle \underline{\m}\rangle )$.

\end{proof}

\begin{corollary}
For any any ordered multisegment $\underline{\m}\in \mathcal{W}(\seg)$, the set $\JJ_+$ is finite.

The dominant part of the limit $q$-character is a polynomial. In other words, we have an additive map
\[
\qchp: \Z[\seg]\to \Z[\seg]\;.
\]
\end{corollary}

\subsection{Main theorems}

We would like to show that the dominant part of the limit $q$-character map coincides with the reciprocal character map for affine Hecke algebras. 

For that goal we will construct a comparison mechanism between the Mackey theory combinatorics for standard modules (as in Proposition \ref{prop:mackey}) and bitableau combinatorics for $\qch$ as in Section \ref{sect:explic}.

\subsubsection{Bitableaux transfer}

Let $\m\in \Z_{\geq0}^{\seg}$ be a given multisegment.

We construct an injective map 
\[
\varphi_{\m}: \JJ_0 \to \mathcal{K}(\underline{\m})_0\;,
\]
for any ordering $\underline{\m}\in \mathcal{W}(\seg)$ of $\m$, where $l_{\m}=|E(\m)|$.

Let us write $E(\m) = \{d_1< \ldots <d_{l_{\m}}\}$, and denote $e(d) =e$ for each $d =d_e\in E(\m)$.

Let $\underline{\m} = [x_1,y_1]\cdots[x_r,y_r]$ be a fixed such ordering of $\m$.

For a given bitableau $\mathcal{P} = \mathcal{P}_1\cdots\mathcal{P}_r\in \JJ_0$ with $\mathcal{P}_i = (a^i_1,b^i_1)\cdots (a^i_{k_i},b^i_{k_i})$, $i=1,\ldots,r$, we define
\[
\varphi_{\m}(\mathcal{P}) = \mathcal{Q}_1\cdots \mathcal{Q}_r\in \tab(\Z\times \Z)
\]
with
\[
\mathcal{Q}_i:= \left\{\begin{array}{ll} (a^i_1,e(b^i_1))\cdots (a^i_{k_i},e(b^i_{k_i}))\,,& a^i_1 < y_i+1 \\ 
(a^i_2,e(b^i_2))\cdots (a^i_{k_i},e(b^i_{k_i}))\,,& a^i_1 = y_i+1 \end{array}  \right.\;.
\]

Put differently, there is a natural embedding of index sets $\iota_{\mathcal{P}}: I_{\varphi_{\m}(\mathcal{P})}\to I_{\mathcal{P}}$ that is given by
\[
\iota_{\mathcal{P}}(i,j) = \left\{ \begin{array}{ll} (i,j) &  a^i_1 < y_i+1 \\ (i,j+1) & a^i_1 = y_i+1 \end{array}  \right.\;,
\]
so that $\varphi_{\m}(\mathcal{P})(i,j) = (a,e(b))$, when $\mathcal{P}(\iota_{\mathcal{P}}(i,j))= (a,b)$.

\begin{proposition}\label{prop:bij0}
\begin{enumerate}
    \item\label{it:bij01}
The map $\varphi_{\m}:\JJ_0 \to \mathcal{K}(\underline{\m})_0$ is well-defined and is a bijection.

\item\label{it:bij02} Whenever $\varphi_{\m}(\mathcal{P}) = \mathcal{Q}$ holds in the bijection above, we have $\iota_{\mathcal{P}}(\widehat{I}_{\mathcal{Q}}\sqcup \widetilde{I}_{\mathcal{Q}}) = \widehat{I}_{\mathcal{P}}$

\item\label{it:bij03} If $\tau: \widehat{I}_{\mathcal{Q}}\sqcup \widetilde{I}_{\mathcal{Q}}\to I_{\mathcal{Q}}$ is a Mackey connection for $\mathcal{Q}$, then $\iota_{\mathcal{P}}\circ\tau\circ \iota_{\mathcal{P}}^{-1}$ is a dominant connection for $\mathcal{P}$.

\item\label{it:bij04} If $\tau': \widehat{I}_{\mathcal{P}}\to I_{\mathcal{P}}$ is a dominant connection for $\mathcal{P}$, there is Mackey connection $\tau$ for $\mathcal{Q}$, such that $\tau' = \iota_{\mathcal{P}}\circ\tau\circ \iota_{\mathcal{P}}^{-1}$.

\end{enumerate}
\end{proposition}
\begin{proof}
Let us write $\mathcal{P} =(a^i_j,b^i_j)_{i,j}\in \JJ_0$. 
\begin{enumerate}
    \item 

Let us note that $d_{e(b^i_j)} = b^i_j$ holds, for all $(i,j)\in I_{\mathcal{P}}$. Thus, the condition $\varphi_{\m}(\mathcal{P})\in \mathcal{K}(\underline{\m})_0$ amounts to either $y_i\leq b^i_1$ or $y_i\leq b^i_2$, for all $1\leq i\leq r$, which is clearly satisfied.

We further see that writing $\psi_{\m}(\mathcal{Q}) = \mathcal{P}_1\cdots \mathcal{P}_r\in \tab(\Z\times \Z)$ with
\[
\mathcal{P}_i:= \left\{\begin{array}{ll} (a^i_1,d_{e^i_1})\cdots (a^i_{k_i},d_{e^i_{s_i}})\,,& y_i= d_{e^i_1} \\ 
(y_i+1, y_i)(a^i_1,d_{e^i_1})\cdots (a^i_{k_i},d_{e^i_{s_i}})\,,& y_i < d_{e^i_1}\end{array}  \right.\;,
\]
for a given $\mathcal{Q} = \mathcal{Q}_1\cdots\mathcal{Q}_r\in \mathcal{K}(\underline{\m})_0$ with $\mathcal{Q}_i = (a^i_1,e^i_1)\cdots (a^i_{s_i},e^i_{s_i})$, $i=1,\ldots,r$, provides an inverse map to $\varphi_{\m}$.

\item This is clear from the definition of the inverse $\psi_{\m}$.

\item Let us note, that for all $(i,j)\in I_{\mathcal{Q}}$, we have $\widehat{\mathcal{Q}}(i,j) = (a,e(b))$, when $\widehat{\mathcal{P}}(\iota_{\mathcal{P}}(i,j))= (a,b)$. Thus, the claim follows directly from definitions, once \eqref{it:bij02} is established.

\item We are left to verify the inclusion $\mathrm{Im}(\tau')\subseteq \mathrm{Im}(\iota_{\mathcal{P}})$ holds, so that $\tau = \iota_{\mathcal{P}}^{-1}\circ \tau'\circ \iota_{\mathcal{P}}$ becomes well-defined.

Indeed, when assuming the contrary we obtain $(i,j)\in \widehat{I}_{\mathcal{P}}$ with $\tau'(i,j) = (i',1)$ and $a^{i'}_1 = y_{i'}+1= b^{i'}_1+1$. Since $\tau'$ is a dominant connection, we have $(a^{i'}_1, b^{i'}_1) = (a^i_{j-1}, b^i_j)$, which now implies $a^i_{j-1}= b^i_j+1$. Yet, $j>1$ gives $a^i_{j-1}\leq y_i+1\leq b^i_j$ and a contradiction.

\end{enumerate}
\end{proof}

\subsubsection{Final comparison}

\begin{theorem}\label{thm:theta-transfer}
Let $\underline{\m}\in \mathcal{W}(\seg)$ be an ordering of a multisegment $\m\in \Z_{\geq0}^{\seg}$. 

The map $\varphi_{\m}$ restricts to a bijection between the sets of tableaux $\JJ_+$ and $\mathcal{K}(\underline{\m})_+$.

An equality of multisegments
\[
\theta(\mathcal{P}) = \theta(\varphi_{\m}(\mathcal{P}))\in \Z_{\geq0}^{\seg}
\]
holds, for all $\mathcal{P}\in \JJ_+$.

\end{theorem}

\begin{proof}
The bijection claim follows from the equivalence of connection notions that is stated in Proposition \ref{prop:bij0} and the criteria of Proposition \ref{prop:crit-mackey-conn} and Theorem \ref{thm:connection}.

For the second claim, let us fix tableaux $\mathcal{P} =(a^i_j,b^i_j)_{i,j}\in  \mathcal{J}(\underline{\m})_+$ and $\mathcal{Q}= \varphi_{\m}(\mathcal{P}) =(\widetilde{a}^i_j,e^i_j)_{i,j}\in \mathcal{K}(\underline{\m})_+$ and connections $\tau' = \iota_{\mathcal{P}}\circ\tau\circ \iota_{\mathcal{P}}^{-1}$ as in Proposition \ref{prop:bij0}\eqref{it:bij03}\eqref{it:bij04}. 

By Proposition \ref{prop:bij0}\eqref{it:bij02}, we have an equality
\[
 \sum_{(i,j)\in I_{\mathcal{P}} \setminus \tau' \left(\widehat{I}_{\mathcal{P}}\right)} [a^i_j, b^i_j]=\sum_{(i,j)\in I_{\mathcal{Q}} \setminus \tau \left(\widehat{I}_{\mathcal{Q}}\sqcup \widetilde{I}_{\mathcal{Q}}\right)} [\widetilde{a}^i_j, d_{e^i_j}] +  \sum_{(i,j)\in I_{\mathcal{P}} \setminus \iota_{\mathcal{P}} \left(I_{\mathcal{Q}}\right)} [a^i_j, b^i_j]
\]
in $\Z_{\geq0}^{\seg}$. The right-hand side of this equality equals $\theta(\mathcal{Q})$ by Proposition \ref{prop:crit-mackey-conn} and the observation that $[a^i_j, b^i_j] = [y^i_1+1, y^i]=0$, for all $(i,j)\in I_{\mathcal{P}} \setminus \iota_{\mathcal{P}} \left(I_{\mathcal{Q}}\right)$. The left-hand side equals to $\theta(\mathcal{P})$ by Theorem \ref{thm:connection}.

\end{proof}

\begin{theorem}\label{thm:main}
For all multisegments $\m,\n\in \Z_{\geq0}^{\seg}$ and an ordering $\underline{\m}$ of $\m$, we have
\[
A(\m,\n) = \#\left\{\mathcal{Q}\in \JJ_+ \;:\; \theta(\mathcal{Q}) = \n \right\} \;.
\]
In particular, we have
\[
\qchp(\exp(\m)) = \sum_{\n\in \Z_{\geq0}^{\seg}} A(\m,\n)\exp(\n)\;,
\]
and the diagram
\[
\begin{diagram}
\dgARROWLENGTH=2em
\node{ \Z[\seg] } \arrow{s,t}{\underline{\zeta} }  \arrow{e,t}{ \qchp }\node{  \Z[\seg] }\\
\node{\mathbf{R_1} } \arrow{ne,t}{ \Large \ch^\otimes } 
\end{diagram}\;
\]
    commutes, with respect to the maps in \eqref{eq:zel-map} and \eqref{eq:recip-def}.
\end{theorem}

\begin{proof}
The first equality follows from Corollary \ref{cor:Kcounting} and Theorem \ref{thm:theta-transfer}. The second equality follows from the first one and the description of $\qch$, and thus of $\qchp$, in Proposition \ref{prop:chq-description}. Finally, the diagram commutes, since 
\[
\ch^{\otimes}(M(\m)) = \sum_{\n\in \Z_{\geq0}^{\seg}} A(\m,\n)\exp(\n)
\]
holds by definition, and $\{\exp(\m)\}_{\m\in \Z_{\geq0}^{\seg}}$ give a basis for $\Z[\seg]$.

\end{proof}

\subsubsection{Proof of Theorem \ref{thm:introD}}

Let $\m,\n\in \Z_{\geq0}^{\seg}$ be fixed multisegments, and $\underline{\m}$ be a fixed ordering of $\m$. 

When combining the bijection of Theorem \ref{thm:theta-transfer} with Proposition \ref{prop:+0}, we see that for each bitableau $(a_j^i , e^i_j)_{i,j}\in \mathcal{K}(\underline{\m})_+$, each entry satisfies $a_j^i\in D(\m)$. Bearing in mind the inclusion $\mathcal{K}(\underline{\m})_+\subseteq \mathcal{K}(\underline{\m})_0$ and writing $E(\m)= \{d_1<\ldots<d_{l_{\m}}\}$, an injective map
\begin{equation}\label{eq:AK}
\mathcal{P}=(a_j^i , e^i_j)_{i,j}\in \mathcal{K}(\underline{\m})_+ \;\mapsto\; \mathcal{P}^\dagger = (a_j^i , d_{e^i_j})_{i,j}\in \mathcal{A}(\m)\;,
\end{equation}
is now visible, where $\mathcal{A}(\m)$ is the set of bitableaux as defined in Section \ref{sect:intro-tabl}.

Moreover, the formula of Proposition \ref{prop:crit-mackey-conn} may be written in the form of 
\[
\theta(\mathcal{P}) = \sum_{\Delta\in \seg} \left( t(\mathcal{P}^\dagger,\Delta) - \overleftarrow{t}(\mathcal{P}^\dagger,\Delta)\right)\Delta\;,
\]
for all $\mathcal{P}\in \mathcal{K}(\underline{\m})_+$.

Thus, the map in \eqref{eq:AK} is a bijection between $\mathcal{K}(\underline{\m})_+$ and $\mathcal{A}(\m)_+$, and the formula of Corollary \ref{cor:Kcounting} is identified with the one in the statement of Theorem \ref{thm:introD}.

\section{Compatibility with quantum affine Schur--Weyl duality}

We now bring the affine Hecke and quantum affine sides together through the quantum affine Schur--Weyl duality functors of \cite{Chari-Pressley}. 

Here we assume that the same parameter $v\in \C^\times$ defines both the quantum affine algebras $U_{N+1}$ and the affine Hecke algebras $H_n=H_{n,v}$.

For every pair of integers $n\geq 1$ and $N\geq 1$, quantum affine
Schur--Weyl duality provides an exact functor
\[
        \mathcal{F}_{N,n}:\Rep(H_n)\longrightarrow \mathcal{C}_N .
\]
We also use the convention that $\mathcal{F}_{N,0}$ sends the unique simple $H_0$-module to the monoidal unit $\mathbf{1}\in\mathcal{C}_N$. 

These functors may be unified into a functor
\[
        \mathcal{F}_N:=\oplus_{n\geq 0}\mathcal{F}_{N,n}:
        \bigoplus_{n\geq 0}\Rep(H_n)\longrightarrow \mathcal{C}_N .
\]
A fundamental property of this construction is its compatibility (\cite[Proposition 4.7]{Chari-Pressley}) with the
monoidal structures on both sides: for $V_i\in \Rep(H_{n_i})$, $i=1,2$, there are natural isomorphisms
\[
        \mathcal{F}_N(V_1\times V_2)\cong \mathcal{F}_N(V_1)\otimes \mathcal{F}_N(V_2).
\]
Thus, $\mathcal{F}_N$ induces a homomorphism $[\mathcal{F}_N]:\mathbf{R}\to  G[N]$ of Grothendieck rings.

The following compatibility describes this homomorphism in the
parametrizations used throughout this work.






\begin{theorem}\label{thm:commute-qasw}
    The diagram of ring homomorphisms
\[
\begin{diagram}
\dgARROWLENGTH=2em
\node{ \Z\left[\widetilde{\seg}\right] } \arrow{s,t}{\widetilde{\underline{\zeta}} }  \arrow{e,t}{ P_N } \node{  \Z[\mathscr{D}_N] }\arrow{s,t}{ S}\\
\node{\mathbf{R} } \arrow{e,t}{ \Large [\mathcal{F}_N]} \node{ G[N]  }
\end{diagram}\;
\]
commutes, where $P_N$ is the quotient map defined in Section \ref{sect:bz-intro}, $\widetilde{\underline{\zeta}}$ is the Zelevinsky isomorphism of Section \ref{sect:zel}, and $S$ is the isomorphism of Theorem \ref{prop:poly-ring} which realized the Drinfeld polynomials classification.

\end{theorem}
\begin{proof}
It is enough to verify the equality on the ring generators $\exp(\Delta)\in  \Z\left[\widetilde{\seg}\right]$, for a segment $\Delta \in \widetilde{\seg}$.

Let us write $\Delta = \alpha_\eta(\Delta_0)$, for $\Delta_0=\left[\frac{p-l+1}{2},\frac{p+l-1}{2}\right]\in \seg$ and $\eta\in \C^\times$. Thus, $\widetilde{\underline{\zeta}}(\exp(\Delta)) = [\sigma_\eta^\ast(Z(\Delta_0)]\in \mathbf{R}_\eta$.

It is known by \cite[Theorem 7.6]{Chari-Pressley}, \cite[Proposition 4.9]{kkk-inentiones} or \cite[Proposition 7]{arakawa-formula} that segment modules are sent to fundamental modules under the functor $\mathcal{F}_N$. More precisely, the cited literature gives the formula
\[
\mathcal{F}_N(\sigma_\eta^\ast(Z(\Delta_0)) \cong \left\{ \begin{array}{ll} L(Y( l, \eta v^p)) & l< N \\ \mathbf{1} & l=N \\ 0 & l>N\end{array}  \right.\;.
\]
In particular, we see that the defining formula for $P_N$ gives $[\mathcal{F}_N(\widetilde{\underline{\zeta}}(\exp(\Delta)))] = S(P_N(\exp(\Delta)))$.

\end{proof}

\begin{corollary}\label{cor:decomp}
The map $[\mathcal{F}_N]$ is surjective.

It satisfies the intertwining relation $\tau_\eta^\ast\circ[\mathcal{F}_N]=[\mathcal{F}_N]\circ \sigma_\eta^\ast$, for all $\eta\in \C^\times$, with respect to twisting by algebra automorphisms on both sides. In particular, $[\mathcal{F}_N](\mathbf{R}_\eta) = G_\eta[N]$ holds, for all $\eta\in \C^\times$.
\end{corollary}

\begin{proof}
We clearly have an equality $P_N\circ \alpha_\eta = \alpha_\eta\circ P_N \circ \alpha_1: \Z[\seg]\to \Z[\mathscr{D}_N]$ of ring homomorphisms.

That equality in combination with the identities \eqref{eq:intertwHeta}, \eqref{eq:intertwUeta} and Theorem \ref{thm:commute-qasw} provides the desired relation.

\end{proof}

\subsection{Proof of Theorem \ref{thm:Aintro}}

Let us fix a module $V\in  \Rep(H_n)$ and set $W = \mathcal{F}_N(V)\in \mathcal{C}_N$. We would like to prove the identity $P_N(\ch^\otimes(V)) = \chi^N(W)_+$ in the ring $\Z[\mathscr{D}_N]$.

Since all of the maps involved factor through the tensor decompositions of Proposition \ref{prop:bzring-decomp}, \eqref{eq:decomp-ring} and \eqref{eq:qa-tensor}, we may assume that $V\in \Rep_\eta(H_n)$, for $\eta\in \C^\times$. Moreover, by Corollary \ref{cor:decomp}, the identity \eqref{eq:qch-twist} and the definition of $\ch^\otimes$ in Section \ref{sect:nonintegral-qch}, we may apply twisting by the algebra automorphism $\sigma_\eta^\ast$ and reduce the problem to the case of $\eta=1$.

Thus, we assume that $[V]\in \mathbf{R}_1$, $[W]\in G_1[N]$ and the equality 
\begin{equation}\label{eq:want}
p_N(\ch^\otimes(V)) = \phi(\chi^N(W))_+
\end{equation}
in the ring $\Z[\seg_{\leq N}]$ remains to be proved.

By the Zelevinsky classification of Section \ref{sect:zel}, we may write $[V]=\underline{\zeta}(X)$, for a polynomial $X\in \Z[\seg]$. By Theorem \ref{thm:main}, we have $\ch^\otimes(V) = \qchp(X)\in \Z[\seg]$.

Lemma \ref{lem:coheren-plus} now implies that 
\begin{equation}\label{eq:have}
p_N(\ch^\otimes(V)) = \qchN(p_N(X))_+\;.
\end{equation}

On the other hand, Theorem \ref{thm:commute-qasw} implies that $S(\phi^{-1}(p_N(X))) = S(P_N(\alpha_1(X))) = [W]$, where $\phi$ is the map from Section \ref{sect:phimap}. Thus, it follows from Theorem \ref{thm:qchar-ident} that 
\begin{equation}\label{eq:have2}
\phi(\chi^N(W)) = \qchN(p_N(X))\;.
\end{equation}
The equality \eqref{eq:want} now follows from \eqref{eq:have} and \eqref{eq:have2}.

\subsection{Proof of Theorem \ref{thm:introC}}

The following is a direct corollary of Theorems \ref{thm:Aintro} and \ref{thm:introB}.

\begin{corollary}\label{cor:last}
Let $V\in \Rep_1(H_n)$ be a given module, and $W = \mathcal{F}_N(V) \in \mathcal{C}_{N,1}$ be its image under the quantum affine Schur--Weyl duality.

Let $a_V: \mathbf{f}\to \mathbb{A}$ be the $\mathbb{A}$-module homomorphism, that is given by $[V] = \Phi(a_V\otimes 1)$, through the categorification isomorphism $\Phi$ of Theorem \ref{thm:ch-quant}.

Then, the formula
\[
\chi^N_+(W) = \sum_{\m\in \Z_{\geq0}^{\seg}} a_V(m_{\m})(1) p_N(\exp(\m)) \in \Z[\mathscr{D}_N]
\]
provides the dominant part of the $q$-character of the module $W$.

\end{corollary}

For proving Theorem \ref{thm:introC}, let us now take a module $W\in \mathcal{C}_N$, which is either standard or simple, and is parametrized by the monomial $p_N(\exp(\m_0))\in \Z[\mathscr{D}_{N,1}]$ with $\m_0\in \Z_{\geq0}^{\seg}$.

The case of standard $W$ means that when taking an admissible ordering $\underline{\m_0}= \Delta_1\cdot \ldots\cdot \Delta_r$ of $\m_0$, the module in hand can be written as the product $W\cong L(p_N(\exp(\Delta_1)))\otimes\cdots \otimes L(p_N(\exp(\Delta_r)))$ of fundamental modules.

Hence, by Theorem \ref{thm:commute-qasw} we have $\mathcal{F}_N(V)\cong W$, when $V= \zeta(\m_0)$ is taken as a standard module of an affine Hecke algebra.

Proposition \ref{prop:kl} now shows that the functional $a_V$ of Corollary \ref{cor:last} is given in that case by the dual basis element to $e_{\m_0}$, in the PBW basis $\mathscr{E} = \{e_{\m}\}_{\m\in \Z_{\geq0}^{\seg}}$ for $\mathbf{f}$. Thus, the formula \eqref{eq:C1} for $\chi^N_+(W)$ follows from that corollary.

As for the case of a simple $W$, the Drinfeld polynomials classification claims that $W$ may be produced as the unique simple quotient module of the standard module $\mathcal{F}_N(\zeta(\m_0))$ that was treated in the former case. Here we deduce that $\mathcal{F}_N(Z(\m_0))\cong W$ must hold.

Thus, $V= Z(\m_0)$ can be chosen, so that $a_V= b_0^\ast$ holds, for a dual canonical basis element, and the formula \eqref{eq:C1} follows again from Corollary \ref{cor:last}.

Finally, we recall that for a multiset $\m^\dagger\in \Z_{\geq0}^{\mathscr{D}_{N,1}}$, the corresponding weight dimension in $W$ is given by the coefficent of the monomial $\exp(\m^\dagger)$ in the polynomial $\chi^N(W)$. Thus, we have
\begin{equation}\label{eq:wtformula}
\dim(W_{\m^\dagger}) = \sum_{\m\in \Z_{\geq0}^{\seg}\;:\; p_N(\exp(\m)) = \exp(\m^\dagger)} a_V(m_{\m})(1)
\end{equation}
from Corollary \ref{cor:last}.

It is a known feature of the setup in Section \ref{sect:qgroups} that the transition matrices in $\mathbf{f}$ between canonical, PBW and monomial bases are all homogeneous, in the sense that $e_{\m_0}^\ast(m_{\n}) = b_0^\ast(m_{\n})=0$ holds, whenever $\supp(\n)\neq \supp(\m_0)$.

Thus, formula \eqref{eq:C2} of Theorem \ref{thm:introC} now follows, in both standard and simple cases, from \eqref{eq:wtformula} and the following lemma.

\begin{lemma}
Let $\m,\m'\in \Z_{\geq0}^{\seg}$ be multisegments with $\supp(\m)= \supp(\m')$ and $p_N(\exp(\m)) = p_N(\exp(\m'))\neq0$, for some $N\geq1$.  Then, $\m=\m'$.
\end{lemma}

\begin{proof}

Let us first assume that $p_N(\exp(\m))=1$, that is, $\m,\m'\in \Z_{\geq0}^{\seg_{N+1}}$, and prove the claim by induction on the height $|\supp(\m)|$.

Let us write $d_0 = \max\{d:\;\supp(\m)(d)\neq0\}$ and $r_0 = \supp(\m)(d_0)$. We clearly have $d_0 = \max E(\m) = \max E(\m')$, and
\[
\m[d_0] = r_0\cdot [d_0- N, d_0] = \m'[d_0]\neq0\;.
\]
The induction hypothesis now implies that the multisegments $\m - \m[d_0], \m'-\m'[d_0]\in \Z_{\geq0}^{\seg_{N+1}}$ are equal, and the claim follows.

Now, let us take the general case of $\m,\m'\in \Z_{\geq0}^{\seg}$. The non-vanishing of $p_N(\exp(\m))$ assumption implies that $\m,\m'\in \Z_{\geq0}^{\seg_{\leq N+1}}$.

Let us decompose $\m = \n + \mathfrak{r}$, $\m' = \n' + \mathfrak{r}'$, so that $\n,\n'\in \Z_{\geq0}^{\seg_{\leq N}} $ and $\mathfrak{r}, \mathfrak{r}'\in \Z_{\geq0}^{\seg_{N+1}}$.

Thus, $p_N(\exp(\n)) = p_N(\exp(\m)) = p_N(\exp(\m')) = p_N(\exp(\n'))$ holds, but $p_N\circ \exp$ is injective on $\Z_{\geq0}^{\seg_{\leq N}}$. Therefore, we have $\n = \n'$, and $\supp(\mathfrak{r}) = \supp(\mathfrak{r}')$. The equality $\mathfrak{r} = \mathfrak{r'}$ now follows from the special case treated earlier.

\end{proof}

\begin{sloppypar} 
\printbibliography[title={References}] 
\end{sloppypar}

\end{document}